\documentclass[11pt,reqno]{amsart}

\usepackage{palatino, euler}
\usepackage{accents,color,enumerate,float,verbatim, mathbbol, bbm}
\usepackage{comment}
\usepackage{graphicx}
\usepackage{amsmath, amsfonts, amssymb, amsthm, amsbsy,amstext}
\usepackage[utf8]{inputenc}
\usepackage[margin=1in]{geometry}
\usepackage{enumitem}
\usepackage[font=small]{caption}
\usepackage{tabularx}
\usepackage{fancyhdr}
\usepackage{hyperref}
\usepackage{tikz}
\usetikzlibrary{angles,quotes}
\usepackage{mdframed}
\usepackage{lipsum}
\usepackage{chngcntr}
\usepackage{xcolor}
\usepackage{imakeidx}
\usepackage{tkz-euclide}
\usepackage{blindtext}
\usepackage[inline]{asymptote}
\usepackage{mathdots}
\usepackage{tikz-cd,mathtools}
\usepackage{ mathrsfs }
\usepackage{float}
\restylefloat{table}

\usepackage{pgfplots, pgfplotstable}
\usepackage{multicol}
\usepackage{endnotes}
\usepackage{idxlayout}
\usepackage{xcolor, forest}
\usepackage{tocvsec2}
\usepackage[toc]{glossaries}
\usepackage{venndiagram}
\usepgfplotslibrary{statistics}
\usepackage[colorinlistoftodos]{todonotes} 
\usepackage[normalem]{ulem}
\usetikzlibrary{shapes.geometric}
\usepackage{mathtools}
\usepackage{standalone}
\usepackage{fancyhdr}
\usepackage[math]{cellspace}
\usepackage{makecell}
\usepackage{xparse}

\setcellgapes{2pt}

\numberwithin{equation}{section}
\newcommand{\R}{\mathbb{R}}

\newcommand{\Z}{\mathbb{Z}}

\newcommand{\stab}{\textrm{stab}}

\newcommand{\Mod}[1]{\ \left(\mathrm{mod}\ #1\right)}
\DeclareMathOperator{\img}{im}
\DeclareMathOperator{\im}{im}

\renewcommand{\Mod}[1]{\ (\mathrm{mod}\ #1)}
\newcommand{\cvector}[3]{\begin{pmatrix} #1 \\ #2 \\ #3 \end{pmatrix}}
\newcommand{\ctwovector}[2]{\begin{pmatrix} #1 \\ #2 \end{pmatrix}}

\ExplSyntaxOn
\NewDocumentCommand{\ccvector}{>{\SplitArgument{2}{,}}m}
 {
  \ccvector_main:nnn #1
 }
\NewDocumentCommand{\ccvector_main:nnn}{mmm}
 {
  \begin{pmatrix}
  #1 \\ #2 \\ #3
  \end{pmatrix}
 }
\ExplSyntaxOff

\newtheorem{thm}{Theorem}[section]

\newtheorem{cor}[thm]{Corollary}
\newtheorem{defn}[thm]{Definition}
\newtheorem{ex}[thm]{Example}
\newtheorem{lem}[thm]{Lemma}

\newtheorem{prop}[thm]{Proposition}
\newtheorem{rem}[thm]{Remark}
\newtheorem{question}[thm]{Question}

\begin{document}

%%%%%%%%%%%%%%%%%%%%%%%%%%%%%%%%%%%%%%%%%%%%%%%%%%%%%
%%%%% Begin setting up title, authors, abstract
%%%%%%%%%%%%%%%%%%%%%%%%%%%%%%%%%%%%%%%%%%%%%%%%%%%%%

\title{Chip-firing and critical groups of signed graphs}

\author{Matthew Cho}
\address{University of California, San Diego}
\email{macho@ucsd.edu}
\author{Anton Dochtermann}
\address{Texas State University}
\email{dochtermann@txstate.edu}
\author{Ryota Inagaki}
\address{Massachusetts Institute of Technology}
\email{inaga270@mit.edu}
\author{Suho Oh}
\address{Texas State University}
\email{suhooh@txstate.edu}
\author{Dylan Snustad}
\address{University of Minnesota, Twin Cities}
\email{snust024@umn.edu}
\author{Bailee Zacovic}
\address{University of Michigan, Ann Arbor}
\email{bzacovic@umich.edu}

\date{\today}

\begin{abstract}
We study chip-firing on a signed graph $G_\phi$, employing a general theory of chip-firing on invertible matrices introduced by Guzm\'an and Klivans. 
Here a negative edge designates an adversarial relationship, so that firing a vertex incident to such an edge leads to a loss of chips at both endpoints.  
The chip-firing rule for $G_\phi$ is described by its reduced Laplacian matrix $L_{G_\phi}$, which also defines the critical group ${\mathcal K}(G_\phi)$. 
The valid chip configurations are given by the lattice points of a rational cone determined by $G_\phi$ and the underlying graph $G$. 
This gives rise to notions of \textit{critical} as well as \textit{$z$-superstable} configurations, both of which are counted by the determinant of $L_{G_\phi}$. We establish general results regarding these configurations, focusing on efficient methods of verifying the underlying properties. We then study the critical groups of signed graphs in the context of vertex switching and Smith normal forms. We use this to compute the critical groups of various classes of signed graphs including signed cycles, wheels, complete graphs, and fans, in the process generalizing results of Biggs and others.
\end{abstract}

\maketitle

\section{Introduction}\label{sec:introduction section}
Classical chip-firing is a single player game played on a connected graph $G$, where `chips' are placed on the vertices of $G$ and distributed to their neighbors via simple `firing' moves. The resulting discrete dynamical system has connections to many areas of algebra, combinatorics, and physics, and has recently found applications in divisor theory in algebraic geometry.  We refer to the recent texts  \cite{CorryPerk, Klivans} for an overview of the subject.

In the version of chip-firing most relevant to us, a single vertex of $G$ is distinguished as the \emph{sink}, and a number of chips are placed on each nonsink vertex. This distribution of chips is represented by an integer vector $\vec{c} \in \mathbb{Z}^{n}$ called a \textit{(chip) configuration}. If a non-sink vertex has at least as many chips as its degree, it can fire, passing one chip to each of its neighbors.  If no non-sink vertex can fire we say that the configuration $\vec{c}$ is \emph{stable}. 
%At this point the sink can fire, distributing a chip to each of its neighbors.
%Repeatedly firing the sink will lead to an unstable configuration.  
If $G$ is a connected graph then any initial configuration eventually stabilizes, as chips are eventually passed to the sink vertex.
If we let $v_1, ..., v_{n}$ denote the nonsink vertices of $G$, we can describe the rules of chip-firing via $L_G$, the $n \times n$ \textit{reduced Laplacian} matrix of $G$.
The result of firing a vertex $v_i$ on a chip configuration $\vec{c}$ is encoded by $\vec{c} - L\vec{e_i}$, where $\vec{e_i}$ is the $i$th standard basis vector. The matrix $L_G$ defines an equivalence relation on the set of vectors in ${\mathbb Z}^{n}$, where $\vec{c}$ and $\vec{d}$ are \emph{firing equivalent} if $\vec{c} - \vec{d}$ is contained in the image of $L_G$. This defines the \emph{critical group} of $G$, given by $\mathcal{K}(G) := {\mathbb Z}^n/\im L_G$.

A configuration $\vec{c}$ is \emph{valid} (or \emph{effective}) if $c_i \geq 0$ for all $i = 1, \dots, n$. One is interested in finding distinguished valid configurations in each equivalence class $[\vec{c}] \in {\mathcal K}(G)$.   On the one hand, it can be shown that each $[\vec{c}]$ contains a unique valid configuration that is stable and \emph{critical}, meaning that it can be reached from a `sufficiently large' configuration $\vec{b}$.
The critical configurations of a graph $G$ can also be used to compute in $\mathcal{K}(G)$, as one can show that the stabilization of the sum of two critical configurations is critical.

On the other hand, one can show that each $[\vec{c}]$ contains a unique valid configuration that is \emph{superstable}, meaning that it is stable under \emph{set-firings}.  A superstable configuration is also the solution of a certain energy minimization problem, and can be seen to coincide with the notion of a \emph{$G$-parking function}.  For a connected graph $G$ there exists a simple bijection between the set of critical configurations and the set of superstable configurations, both of which are in bijection with the set of spanning trees of $G$.  

In recent years, chip-firing has been extended to more general settings, where the reduced Laplacian of a graph is replaced by other distribution matrices (see for instance \cite{gabrielov} and \cite{GuzKliMmatrices}). One uses such a matrix $M$ to define a firing rule that mimics the graphical setting: firing a `site' $v_i$ now takes a configuration $\vec{c}$ to $\vec{c} - M \vec{e_i}$. For a well-defined notion of chip-firing we require that $M$ satisfies an \emph{avalanche finite} property, so that repeated firings of any initial configuration eventually stabilize in an appropriate sense.  The class of matrices with this property are known as \emph{$M$-matrices}, and can be characterized in a number of ways (see Definition \ref{def:Mmatrix} below). In \cite{GuzKliMmatrices}, Guzm\'an and Klivans have shown that the chip-firing theory defined by an $M$-matrix leads to good notions of critical and superstable configurations.  An important subtlety here is that one must consider \emph{multiset} firings and the notion of `$z$-superstability', we review these concepts below.

\subsection{Our contributions.}
In this paper we study chip-firing on a \emph{signed graph} $G_\phi$, where each edge of a graph $G$ is assigned a positive or negative weight according to a function $\phi: E(G) \rightarrow \{+,-\}$. When a vertex fires, if one of its incident edges is negative, both that vertex and its neighbor lose a chip. Such a setup can be thought of as modeling a process where both antagonistic (negative) and cooperative (positive) relations exist. A signed graph $G_\phi$ comes with a (reduced) signed Laplacian matrix $L_{G_{\phi}}$ which, as above,  defines a chip-firing rule and also a critical group ${\mathcal K}(G_\phi)$.
%The cardinality of ${\mathcal K}(G_\phi)$ can be understood in terms of combinatorial objects that generalize spanning trees, see Theorem \ref{thm:Zaslavsky}.

The matrix $L_{G_\phi}$ is not necessarily an $M$-matrix (for instance, it will typically have positive entries off the diagonal), and hence does not come with a pre-existing notion of chip-firing that relates to its critical group (as far as we know, see also Section \ref{sec: other}).  However, signed graphs provide a natural setting to employ a more general theory of chip-firing developed by Guzm\'an and Klivans in \cite{GuzKlivans}. 
In this setup, one fixes a \emph{chip-firing pair} $(L,M)$ of $n \times n$ matrices, where $L$ is invertible and $M$ is an $M$-matrix. One uses $L$ to define the chip-firing rule, in the sense that `firing a vertex' is defined by subtracting a column of $L$. An important feature of this theory is that the class of \emph{valid} configurations is no longer simply integer vectors with nonnegative entries, but rather integer points in a cone $S^+$ determined by $LM^{-1}$.  Among the vectors in $S^+$ some are distinguished as $z$-superstable, some as critical. The definitions mimic the setup for graphs, and one can show that many of the desirable properties extend to this setting, we refer to Section \ref{sec:chipfiring pairs} for details.

Since $M$ is itself an $M$-matrix, one can also consider the chip-firing rule determined by $M$, where valid configurations are given by nonnegative integral vectors ${\mathbb Z}^n_{\geq 0}$. 
Our first result provides a way to compare $z$-superstable and critical configuration in these two settings. We emphasize that this result applies for any chip-firing pair $(L,M)$, not necessarily arising from graphs. We do require that $L$ and $M$ are both integral (in which case we refer to $(L,M)$ as an \emph{integral chip-firing pair}).

\newtheorem*{thm:KeyIdea}{Theorem \ref{thm:KeyIdea}}
\begin{thm:KeyIdea}
Suppose $(L,M)$ is a integral chip-firing pair, where $L$ is an invertible matrix and $M$ is an $M$-matrix. Suppose $\vec{c} \in S^+$ is a valid configuration. Then $\vec{c}$ is $z$-superstable if and only if $\lfloor ML^{-1} \vec{c} \rfloor$ is $z$-superstable for $M$. Similarly $\vec{c}$ is critical if and only if $\lfloor ML^{-1} \vec{c} \rfloor$ is critical for $M$.
\end{thm:KeyIdea}

We now wish to apply these ideas to the setting of signed graphs. For a signed graph $G_{\phi}$ with specified sink vertex $q$, we let $G$ denote the underlying (unsigned) graph.  To define a chip-firing pair $(L,M)$, we take $L = L_{G_{\phi}}$ to be the reduced signed Laplacian, and choose $M = L_G$ to be the reduced Laplacian of $G$ (where the reduction is taken relative to $q$). Results of \cite{GuzKlivans} tell us that the pair $(L,M)$ gives rise to notions of critical and $z$-superstable configurations, unique to each equivalence class determined by $L$ (see Definitions \ref{def:critical} and \ref{def:superstable}). Hence the number of such configurations is given by $\det L$, which can be counted by combinatorial objects that generalize spanning trees (see Theorem \ref{thm:Zaslavsky}). This approach to chip-firing on signed graphs was used by Moore in his thesis \cite{Ram}.

Recall that the valid configurations for the pair $(L,M)$ are integer points in the rational cone determined by $LM^{-1}$ (see Definition \ref{defn:valid}). Our next result says that, for the case of a signed graph, these vectors will always have nonnegative entries.

\newtheorem*{prop:positive}{Proposition \ref{prop:positive}}
\begin{prop:positive}
    For any signed graph $G_\phi$ with chosen sink $q$, we have $S^+ \subset {\mathbb Z}^n_{\geq 0}$, so that the set of valid configurations determined by the chip-firing pair $(L,M)$ is contained in the nonnegative orthant.
\end{prop:positive}

Next we turn to special configurations of a signed graph $G_\phi$. 
%By definition, $\vec{c}$ is $z$-superstable if and only if $ML^{-1} \vec{c} - M\vec{z} \not\geq \vec{0}$ for any nonzero vector $\vec{z}$ in ${\mathbb Z}^n_{\geq 0}$.
By definition, to check whether a given valid configuration $\vec{c} \in S^+$ is $z$-superstable, one must check firings of all nonempty \emph{multisets} of vertices. Hence a priori we have no algorithm to check for this property. However, we show that for the case of signed graphs, simple set (as opposed to multiset) firings suffice to check for $z$-superstability.

\newtheorem*{thm:zchigeneral}{Theorem \ref{thm:zchigeneral}}
\begin{thm:zchigeneral}
Suppose $G_{\phi}$ is a signed graph with sink $q$ and chip-firing pair $(L,M)$.  Then a valid configuration $\vec{c} \in S^+$ is $z$-superstable if and only if $ML^{-1}\vec{c} - M\vec{z} \not\geq \vec{0}$ for any nonzero vector $\vec{z}$ in $\{0, 1\}^{n}$. 
\end{thm:zchigeneral}

In the language of chip-firing, Theorem \ref{thm:zchigeneral} says that a configuration is $z$-superstable if and only if it is \emph{$\chi$-superstable}. 
We remark that Theorem \ref{thm:zchigeneral} can also be proved as a Corollary of Theorem \ref{thm:KeyIdea} by employing \cite[Theorem 4.4]{GuzKlivans}. Here we include a self-contained proof, in part to clarify some notational ambiguity in \cite{GuzKlivans}.
From this we obtain the following corollary to Theorem \ref{thm:KeyIdea}.

\newtheorem*{cor:KeyIdea2}{Corollary \ref{cor:KeyIdea2}}
\begin{cor:KeyIdea2}
Let $G_\phi$ be a signed graph with sink $q$ and underlying graph $G$. Suppose $\vec{c} \in S^+$ is a valid configuration. Then $\vec{c}$ is $z$-superstable if and only if $\lfloor ML^{-1} \vec{c} \rfloor$ is superstable for $G$. Similarly, $\vec{c}$ is critical if and only if $\lfloor ML^{-1} \vec{c} \rfloor$ is critical for $G$. 
\end{cor:KeyIdea2}

From the definition of $z$-superstable and critical configurations, it is not immediately clear how one would construct such configurations for a given signed graph $G_\phi$. However, Corollary \ref{cor:KeyIdea2} provides a way to check whether a given configuration $\vec{c} \in S^+$ is $z$-superstable (resp. critical).  It also leads to an algorithm for generating all $z$-superstable and critical configurations for a given signed graph $G_\phi$ with chosen sink $q$, answering a question posed in \cite{Ram}. We refer to Section \ref{sec:check} and Proposition \ref{prop:algorithm} for details.

%\begin{prop}\label{prop:positive}
%    For any signed graph $G_\phi$, the set of valid configurations determined by the chip-firing pair $(L,M)$ is contained in the nonnegative orthant.
%\end{prop}

We next further investigate the critical configurations of signed graphs. In classical chip-firing on graphs, one has a number of results to efficiently detect whether a configuration is critical.  For example, \cite[Theorem 2.6.3]{Klivans} says that a valid configuration $\vec{c}$ is critical if and only if the stabilization after firing the sink recovers $\vec{c}$. 
We can use our results to extend this to the setting of signed graphs. For a signed graph $G_\phi$, we let $\vec{s}$ denote the configuration obtained by firing the sink once (see Section \ref{sec:SinkFire}).
Applying Theorem \ref{thm:KeyIdea} we get the following criteria for checking criticality in the context of signed graphs.

\newtheorem*{prop:EOfSinkFiring0}{Proposition \ref{prop:EOfSinkFiring0}}
\begin{prop:EOfSinkFiring0}
Suppose $G_\phi$ is a signed graph with chip-firing pair $(L,M)$. Then a configuration 
$\vec{c} \in S^+$ is critical if and only if $\mbox{stab}_{S^+}(\vec{c}+LM^{-1}\vec{s}) = \vec{c}$.
\end{prop:EOfSinkFiring0}

Proposition  \ref{prop:EOfSinkFiring0} in particular implies that $LM^{-1} \vec{s}$ is always valid, and in fact has a simple combinatorical interpretation (see Lemma \ref{lem:sink}).   One can also ask whether $\vec{s}$ itself lies in $S^+$, i.e. whether firing the sink of $G_\phi$ gives a valid configuration. 
For the case of a signed graph $G_\phi$ it is in general difficult to decide whether $\vec{s}$ is valid. However, under certain conditions we can guarantee this property, see Corollary \ref{cor:SpecialSinkFire} for details.

\begin{comment}
\begin{rem}
As we have seen,  $LM^{-1} \vec{s}$ is always valid as a configuration in $S^+$.  One can also ask whether $\vec{s}$ itself lies in $S^+$, i.e. whether firing the sink of $G_\phi$ gives a valid configuration. 
%For an unsigned graph $G$, firing the sink always leads to a valid (i.e. nonnegative) configuration.  
For the case of a signed graph $G_\phi$ it is in general difficult to decide whether $\vec{s}$ is valid (see the examples in Section \ref{sec:Examples}). However, if $G_\phi$ has a \emph{universal sink} (a sink adjacent to all other vertices) $q$ and is \emph{signed-regular} (the number of negative and positive edges incident to $v$ is the same for all nonsink vertices $v$) then one can guarantee this property.  We refer to Corollary \ref{cor:SpecialSinkFire} for details.
\end{rem}

\end{comment}

In the context of classical chip-firing on a graph $G$, the `maximal critical' configuration $\vec{c}_{\text{max}}$, defined by  $(\vec{c}_{\text{max}})_i = \deg(v_i) - 1$, plays an important role. In particular $\vec{c}_{\text{max}}$ is a valid critical configuration such that $\vec{c} \leq \vec{c}_{\text{max}}$ (componentwise) for any stable configuration $\vec{c}$.  Furthermore, a configuration $\vec{c}$ is critical if and only if $\vec{c}_{\text{max}} - \vec{c}$ is $z$-superstable.
%This bijection also gives a polynomial time way to check whether a given configuration is critical and/or superstable. 
As shown in \cite{GuzKliMmatrices}, these correspondences generalize in an expected way to the setting of chip-firing on $M$-matrices.
Although these concepts do not seem to extend to general signed graphs, we do recover the notion of a ``maximal critical configuration'' for certain \emph{signed-regular graphs}. We refer to Corollary \ref{cor:MaxCriticalConfig} for details.

\subsection{Critical groups of signed graphs}
In the second part of our paper we study (and compute) critical groups of signed graphs. Recall that  if $G_\phi$ is a signed graph with specified sink $q$ and nonsink vertices $\{v_1, \dots, v_n\}$, its reduced Laplacian $L = L_{G_\phi}$ defines the critical group ${\mathcal K}(G_\phi) := {\mathbb Z}^n/L$. 
In Theorem \ref{thm:CriticalityIdentity} we provide a way to compute the identity of ${\mathcal K}(G_\phi)$ in terms of the identity of ${\mathcal K}(G_\phi)$. In Lemma \ref{lem:SmithBalanced} we show that critical groups are invariant under \emph{vertex switching}, an operation on a signed graph that plays an important role in the theory (see Definition \ref{def:switching}).

We use these results to calculate the critical groups for several classes of signed graphs (with choice of sink $q$), see Figure \ref{fig:AAA} for an illustration of some of the graphs that we consider.

\begin{figure}[h]
 \begin{center}

%  \vspace{3.7mm}
    \input{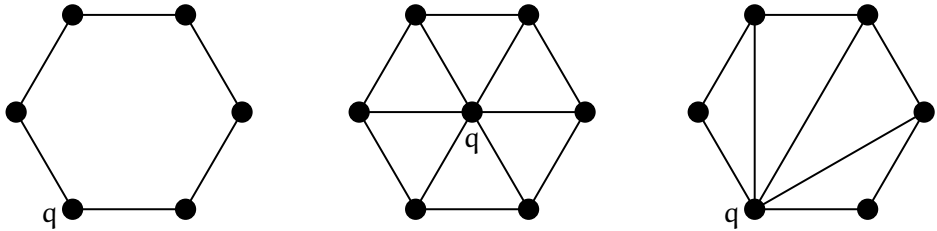}
    \captionsetup{width=1.0\linewidth}
  \captionof{figure}{The Cycle $C_6,$ the Wheel $F_6,$ and the Fan $F_6$.}
  \label{fig:AAA}

 \end{center}
 \end{figure}
 
From Lemma \ref{lem:SmithBalanced} we see that if $G_\phi$ is a connected signed graph with sink $q$ such that $G \backslash q$ is a tree, then there exists an isomorphism of critical groups
\[{\mathcal K}(G_\phi) \cong {\mathcal K}(|G|) \cong {\mathcal K}(G).\]
From this we observe that for any signed cycle $(C_{n})_\phi$, and any choice of sink $q$, we have $\mathcal{K}((C_{n})_\phi) \cong \mathbb{Z}_{n}$, see Proposition~\ref{prop:Cyclecrit}. We also observe a duality between the $z$-superstable and critical configurations for certain signed cycles, see Theorem~\ref{thm:PetitDuality}.

We next consider complete graphs.  The critical group of the unsigned graph $K_n$ is well-known, and although understanding critical groups of arbitrary \emph{signed} complete graphs seems difficult, we are able to compute them for a certain class. In what follows we let $-K_n$ denote the signed complete graph with sink $q$ such that all edges not incident to $q$ have a negative sign.

\newtheorem*{prop:Completecrit}{Proposition \ref{prop:Completecrit}}
\begin{prop:Completecrit}
Suppose $n \geq 2$ and let $(K_{n})_\phi$ be a signed complete graph with sink $q$ that is switching equivalent to $-K_{n}$. 
Then 
\[{\mathcal K}((K_{n})_\phi) \cong \mathbb{Z}_{n-2}^{n-3} \oplus  \mathbb{Z}_{(n-2)(2n-3)}.\]
\end{prop:Completecrit}

We next consider signed wheels, where the sink vertex is taken to be the center (hub) vertex.  Here we can completely describe the possible critical groups, providing a generalization of results of Biggs \cite{Biggs}. In what follows, we let $f_n$ and $\ell_n$ denote the $n$th Fibonacci and $n$th Lucas numbers, respectively (see Definitions \ref{defn:Fib} and \ref{defn:Luc}).

\newtheorem*{thm:wheelresult}{Theorem \ref{thm:wheelresult}}
\begin{thm:wheelresult}
Let  $(W_n)_{\phi}$ be any signed wheel with the hub vertex as the sink.
Then for $n\geq 3$ we have:
$$\mathcal{K}((W_n)_{\phi}) = \begin{cases}
\mathbb{Z}_{f_n} \oplus \mathbb{Z}_{5f_n}, & \mbox{$n$ is odd and $(W_n)_{\phi}$ is unbalanced} \\
\mathbb{Z}_{\ell_n} \oplus \mathbb{Z}_{\ell_n}, & \mbox{$n$ is even and $(W_n)_{\phi}$ is unbalanced}\\\mathbb{Z}_{\ell_n} \oplus \mathbb{Z}_{\ell_n}, & \mbox{$n$ is odd and $(W_n)_{\phi}$ is balanced}\\
\mathbb{Z}_{f_n} \oplus \mathbb{Z}_{5f_n}, & \mbox{$n$ is even and $(W_n)_{\phi}$ is balanced} \end{cases}$$ 
\end{thm:wheelresult}

Finally, we compute the critical groups of signed fans where the sink is the ``handle'' (see Figure \ref{fig:AAA}). Here we again obtain a complete answer.

\newtheorem*{thm:FanGroup}{Theorem \ref{thm:FanGroup}}
\begin{thm:FanGroup}
Let $(F_n)_{\phi}$ be any signed fan graph with sink given by the vertex of degree $n-1$. Then for $n \geq 1$ we have $\mathcal{K}((F_n)_{\phi}) \cong \mathbb{Z}_{f_{2n}}$.
\end{thm:FanGroup}

\subsection{Example} \label{sec:example}
We end this introduction with a worked example to illustrate our constructions and results, and also to emphasize how our theory differs from chip-firing on ordinary graphs. We consider the signed graph $G_\phi$ with chosen sink $q$ as depicted in Figure \ref{fig:chinotz}. For this graph the relevant matrices $L = L_{G_\phi}$, $M = L_G$, and $LM^{-1}$ are given below.
\[L = \begin{pmatrix} 3 & +1 & -1 \\ +1 & 2 & -1 \\ -1 & -1 & 3 \end{pmatrix}; \quad \quad M = \begin{pmatrix} 3 & -1 & -1 \\ -1 & 2 & -1 \\ -1 & -1 & 3 \end{pmatrix}; \quad \quad LM^{-1} = \begin{pmatrix} 2 & 2 & 1 \\ \frac{5}{4} & 2 & \frac{3}{4} \\ 0 & 0 & 1 \end{pmatrix} \]
\begin{figure}[h]
 \begin{center}
%  \vspace{3.7mm}
    \tikzset{every picture/.style={line width=0.75pt}} %set default line width to 0.75pt        

\begin{tikzpicture}[x=0.65pt,y=0.65pt,yscale=-1,xscale=1]

%Straight Lines [id:da7276885163184901] 
\draw [color={rgb, 255:red, 65; green, 117; blue, 5 }  ,draw opacity=1 ][line width=1.5]    (340.62,185.84) -- (256.02,100.88) ;
%Straight Lines [id:da07095566524149133] 
\draw [color={rgb, 255:red, 65; green, 117; blue, 5 }  ,draw opacity=1 ][line width=1.5]    (256.02,100.88) -- (256.02,185.84) ;
%Straight Lines [id:da5917296703412491] 
\draw [color={rgb, 255:red, 65; green, 117; blue, 5 }  ,draw opacity=1 ][line width=1.5]    (340.62,100.88) -- (340.62,185.84) ;
%Straight Lines [id:da267250031806753] 
\draw [color={rgb, 255:red, 208; green, 2; blue, 27 }  ,draw opacity=1 ][line width=1.5]    (256.02,100.88) -- (340.62,100.88) ;
%Straight Lines [id:da507934250444043] 
\draw [color={rgb, 255:red, 65; green, 117; blue, 5 }  ,draw opacity=1 ][line width=1.5]    (256.02,185.84) -- (340.62,185.84) ;
%Shape: Ellipse [id:dp1709730539751455] 
\draw  [fill={rgb, 255:red, 0; green, 0; blue, 0 }  ,fill opacity=1 ] (251.01,100.88) .. controls (251.01,98.07) and (253.25,95.79) .. (256.02,95.79) .. controls (258.79,95.79) and (261.03,98.07) .. (261.03,100.88) .. controls (261.03,103.69) and (258.79,105.97) .. (256.02,105.97) .. controls (253.25,105.97) and (251.01,103.69) .. (251.01,100.88) -- cycle ;
%Shape: Ellipse [id:dp5513613395838337] 
\draw  [fill={rgb, 255:red, 0; green, 0; blue, 0 }  ,fill opacity=1 ] (335.61,100.88) .. controls (335.61,98.07) and (337.86,95.79) .. (340.62,95.79) .. controls (343.39,95.79) and (345.63,98.07) .. (345.63,100.88) .. controls (345.63,103.69) and (343.39,105.97) .. (340.62,105.97) .. controls (337.86,105.97) and (335.61,103.69) .. (335.61,100.88) -- cycle ;
%Shape: Ellipse [id:dp5316281537425069] 
\draw  [fill={rgb, 255:red, 0; green, 0; blue, 0 }  ,fill opacity=1 ] (251.01,185.84) .. controls (251.01,183.03) and (253.25,180.75) .. (256.02,180.75) .. controls (258.79,180.75) and (261.03,183.03) .. (261.03,185.84) .. controls (261.03,188.65) and (258.79,190.93) .. (256.02,190.93) .. controls (253.25,190.93) and (251.01,188.65) .. (251.01,185.84) -- cycle ;
%Shape: Ellipse [id:dp3628163756207392] 
\draw  [fill={rgb, 255:red, 0; green, 0; blue, 0 }  ,fill opacity=1 ] (335.61,185.84) .. controls (335.61,183.03) and (337.86,180.75) .. (340.62,180.75) .. controls (343.39,180.75) and (345.63,183.03) .. (345.63,185.84) .. controls (345.63,188.65) and (343.39,190.93) .. (340.62,190.93) .. controls (337.86,190.93) and (335.61,188.65) .. (335.61,185.84) -- cycle ;

% Text Node
\draw (242.37,194.33) node [anchor=north west][inner sep=0.75pt]    {$q$};
% Text Node
\draw (343.57,136.21) node [anchor=north west][inner sep=0.75pt]    {$+$};
% Text Node
\draw (238.08,136.21) node [anchor=north west][inner sep=0.75pt]    {$+$};
% Text Node
\draw (292.12,189.27) node [anchor=north west][inner sep=0.75pt]    {$+$};
% Text Node
\draw (291.15,84.04) node [anchor=north west][inner sep=0.75pt]    {$-$};
% Text Node
\draw (296.01,125.71) node [anchor=north west][inner sep=0.75pt]    {$+$};
% Text Node
\draw (239.37,79.0) node [anchor=north west][inner sep=0.75pt]    {$v_{1}$};
% Text Node
\draw (338,80.0) node [anchor=north west][inner sep=0.75pt]    {$v_{2}$};
% Text Node
\draw (338.62,194.33) node [anchor=north west][inner sep=0.75pt]    {$v_3$};

\end{tikzpicture}
    \captionsetup{width=1.0\linewidth}
  \captionof{figure}{ An example of a signed graph $G_{\phi}$ where $|{\mathcal K}(G_{\phi})|$ is strictly larger than the number of spanning trees of the underlying graph.}
  \label{fig:chinotz}
 \end{center}
\end{figure}
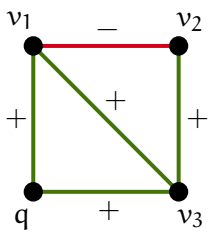

For the underlying graph $G$ (with sink $q$), we have $\det M = 8$, and hence we expect 8 critical (resp. superstable) configurations corresponding to the 8 spanning trees of $G$. Indeed, a calculation shows that the critical configurations on $G$ are
$$\ccvector{2,1,2}, \ccvector{1,1,2}, \ccvector{2,0,2}, \ccvector{2,1,1}, \ccvector{0,1,2}, \ccvector{2,1,0}, \ccvector{1,0,2} \ccvector{2,0,1},$$ 
whereas the superstable configurations on $G$ are 
$$\ccvector{0,0,0}, \ccvector{1,0,0}, \ccvector{0,1,0}, \ccvector{0,0,1}, \ccvector{2,0,0}, \ccvector{0,0,2}, \ccvector{1,1,0} \ccvector{0,1,1}.$$

Considering now the signed graph $G_\phi$, one can check that $\det L = 12$ and hence in this case $|{\mathcal K}(G_\phi)|$ is strictly larger than the number of spanning trees of the underlying graph $G$. Integral row operations can be used to reduce $L$ to Smith normal form to show that $\mathcal{K}(G_{\phi}) \cong \mathbb{Z}_{12}$.
Hence we expect 12 critical configurations and 12 $z$-superstable configurations for $G_\phi$, one for each equivalence class in ${\mathcal K}(G_\phi)$. Indeed we find that the critical configurations of the \emph{signed} graph $G_\phi$ are $$\left\{ \ccvector{7,6,2}, \ccvector{8,6,2}, \ccvector{8,6,1}, \ccvector{6,5,2}, \ccvector{7,5,1}, \ccvector{9,7,0}, \ccvector{6,4,2}, \ccvector{7,5,2}, \ccvector{9,7,2}, \ccvector{9,7,1}, \ccvector{8,6,0}, \ccvector{6,4,1} \right\}.$$ 

A calculation shows that $\begin{pmatrix} 7 \\ 6 \\ 2 \end{pmatrix}$, $\begin{pmatrix} 8 \\ 6 \\ 1 \end{pmatrix}$, $\begin{pmatrix} 9 \\ 7 \\ 0 \end{pmatrix}$, and $\begin{pmatrix} 6 \\ 4 \\ 1 \end{pmatrix}$ are the distinct generators of ${\mathcal K}(G_\phi)$.  
%For instance one can check
%\[\begin{pmatrix} 6 \\ 4 \\ 1 \end{pmatrix} \oplus \begin{pmatrix} 6 \\ 4 \\ 1 \end{pmatrix} = \stab\left(\begin{pmatrix} 12 \\ 8 \\ 2 \end{pmatrix}\right) = \begin{pmatrix} 6 \\ 5 \\ 2 \end{pmatrix}.\]

The $z$-superstable configurations of $G_\phi$ turn out to be $$\left\{ \ccvector{7,5,0}, \ccvector{2,2,0}, \ccvector{5,4,0}, \ccvector{6,4,0}, \ccvector{4,3,0}, \ccvector{5,4,2}, \ccvector{0,0,0}, \ccvector{1,1,0}, \ccvector{3,3,0}, \ccvector{6,5,0}, \ccvector{4,3,2}, \ccvector{3,2,0} \right\}.$$ 

We have listed these $z$-stable configurations so that they are in the same equivalence class as the corresponding critical configurations listed above.
As an illustration of Corollary \ref{cor:KeyIdea2}, we consider the inverse images of these critical configurations under $LM^{-1}$ (i.e. we apply $ML^{-1}$), to obtain the following vectors with rational coordinates:
$$\left\{ \ccvector{\frac{2}{3},\frac{11}{6},2}, \ccvector{2,1,2}, \ccvector{\frac{7}{3},\frac{7}{6},1}, \ccvector{\frac{2}{3},\frac{4}{3},2}, \ccvector{\frac{7}{3},\frac{2}{3},1}, \ccvector{\frac{8}{3},\frac{11}{6},0}, \ccvector{2,0,2}, \ccvector{2,\frac{1}{2},2}, \ccvector{2,\frac{3}{2},2}, \ccvector{\frac{7}{3},\frac{5}{3},1}, \ccvector{\frac{8}{3},\frac{4}{3},0}, \ccvector{\frac{7}{3},\frac{1}{6},1} \right\}.$$ 

Similarly, applying $ML^{-1}$ to the superstable configurations yield:
$$\left\{ \ccvector{\frac{8}{3},\frac{5}{6},0}, \ccvector{0,1,0}, \ccvector{\frac{4}{3},\frac{7}{6},0}, \ccvector{\frac{8}{3},\frac{1}{3},0}, \ccvector{\frac{4}{3},\frac{2}{3},0}, \ccvector{\frac{2}{3},\frac{5}{6},2}, \ccvector{0,0,0}, \ccvector{0,\frac{1}{2},0}, \ccvector{0,\frac{3}{2},0}, \ccvector{\frac{4}{3},\frac{5}{3},0}, \ccvector{\frac{2}{3},\frac{1}{3},2}, \ccvector{\frac{4}{3},\frac{1}{6},0} \right\}.$$ 

According to Corollary \ref{cor:KeyIdea2}, any vector obtained by taking the floor of each entry will result in a critical (resp. superstable) configuration of the underlying graph $G$. For instance taking the floor of the critical configuration $\ccvector{\frac{2}{3},\frac{11}{6},2}$ yields $\ccvector{0,1,2}$, which indeed is critical for $G$.

From Proposition \ref{prop:EOfSinkFiring0}, we have another way of checking for criticality. For this note that $\vec{s} = \ccvector{1,0,1}$, so that $LM^{-1} \vec{s} = \ccvector{3,2,1}$. Now for instance, to verify the criticality of $\ccvector{7,6,2}$ one can check that 
$$\mbox{stab}_{S^+}\left(\ccvector{7,6,2} + \ccvector{3,2,1}\right) = \mbox{stab}_{S^+}\left(\ccvector{10,8,3}\right) = \ccvector{7,6,2}.$$

We emphasize that (as opposed to the unsigned setting) our choice of sink $q$ determines much of the structure discussed above, and in particular even the cardinality of the critical group can change.  For example, if we consider the same signed graph $G_\phi$ as above with sink $q = v_3$, the relevant matrices become

\[L = \begin{pmatrix} 2 & -1 & 0 \\ -1 & 3 & +1 \\ 0 & +1 & 2 \end{pmatrix}; \quad \quad M = \begin{pmatrix} 2 & -1 & 0 \\ -1 & 3 & -1 \\ 0 & -1 & 2 \end{pmatrix}; \quad \quad LM^{-1} = \begin{pmatrix} 1 & 0 & 0 \\ \frac{1}{4} & \frac{3}{2} & \frac{5}{4} \\ \frac{1}{2} & 1 & \frac{3}{2} \end{pmatrix}. \]

In this case $\det L = 8$, so for this choice of sink we get 8 critical and 8 superstable configuration. One can check that the resulting critical group is isomorphic to ${\mathbb Z}_8$.

\subsection{Organization}
The rest of the paper is organized as follows. In Section \ref{sec:prelimsection}, we review the basics of signed graph theory and also summarize the Guzm\'an-Klivans theory of chip-firing pairs. Here we also describe the setup for our study of chip-firing on signed graphs, and discuss other approaches from the literature.  In Section \ref{sec:generalizedResults}, we provide proofs of our main stability results discussed above.  In Section \ref{sec:groups} we study critical groups of signed graphs. We establish some basic properties and compute ${\mathcal K}(G)$ for several classes signed graphs including cycles, wheels, complete graphs, and fans. In Section \ref{sec:Further}, we discuss some open questions and potential avenues for future work.

\section{Preliminaries}\label{sec:prelimsection}

In this section we recall some basic notions from the theory of signed graphs. We assume the reader is familar with basic notions from graph theory. We also review the theory of chip-firing on invertible matrices that will be employed in our study.

\subsection{Signed graphs}

The systematic study of signed graphs was initiated by Zaslavsky in \cite{Zaslavsky}.  We recall some basics from the theory, and refer to \cite{Zaslavsky} for details.  By a \emph{simple graph} we mean a graph with no loops or multiple edges.

\begin{defn}
A signed graph $G_{\phi}$ consists of a simple graph $G =(V,E)$ equipped with a sign function $\phi : E \rightarrow \{+, -\}$. We let $|G|$ denote the signed graph with all positive edges. 
\end{defn}

Suppose $G_\phi$ is a signed graph with vertex set $V = \{v_1, \dots, v_n, v_{n+1}\}$. The \emph{Laplacian} $\tilde{L} = \tilde{L}_{G_\phi}$ is the $(n+1) \times (n+1)$ matrix given by
\[(\tilde{L})_{i, j}   = 
\begin{cases}
\deg(v_i), & \text{ if $i = j$} \\
-1, & \text{ if $\{v_i,v_j\} \in E$ and $\phi(\{v_i,v_j\}) = 1$} \\
1, & \text{ if $\{v_i, v_j\} \in E$ and $\phi(\{v_i,v_j\}) = -1$}
\end{cases}
\]
for $i, j = 1, 2, ..., n+1$. 

The matrix $\tilde{L}$ can also be obtained as $\tilde{L} = \delta \delta^T$, where $\delta$ is the \emph{signed incidence matrix} associated to $G_\phi$.
If we specify a sink vertex $q$ (without loss of generality assume $q = v_{n+1})$, the \emph{reduced Laplacian} $L = L_{G_\phi}$ is the $n \times n$ matrix obtained by removing the row and column corresponding to $q$.  This convention conflicts with other notations from the literature; here we use $L$ to denote the reduced Laplacian since it will be our main object of study.

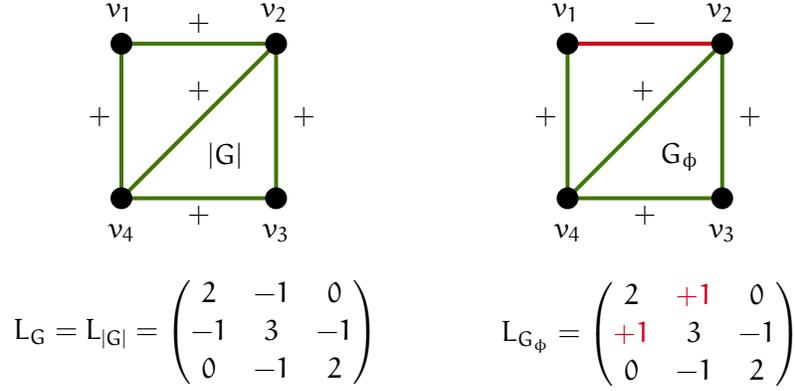
\begin{figure}
 \begin{center}
 
%  \vspace{3.7mm}
    \tikzset{every picture/.style={line width=0.75pt}} %set default line width to 0.75pt        

\begin{tikzpicture}[x=0.75pt,y=0.75pt,yscale=-1,xscale=1]

%Shape: Square [id:dp5736117119202822] 
\draw  [color={rgb, 255:red, 65; green, 117; blue, 5 }  ,draw opacity=1 ][line width=1.5]  (165,64) -- (243,64) -- (243,142) -- (165,142) -- cycle ;
%Straight Lines [id:da6802279645604477] 
\draw [color={rgb, 255:red, 65; green, 117; blue, 5 }  ,draw opacity=1 ][line width=1.5]    (243,64) -- (165,142) ;
%Shape: Circle [id:dp3066694413553579] 
\draw  [fill={rgb, 255:red, 0; green, 0; blue, 0 }  ,fill opacity=1 ] (160,64) .. controls (160,61.24) and (162.24,59) .. (165,59) .. controls (167.76,59) and (170,61.24) .. (170,64) .. controls (170,66.76) and (167.76,69) .. (165,69) .. controls (162.24,69) and (160,66.76) .. (160,64) -- cycle ;
%Shape: Circle [id:dp7026513481154146] 
\draw  [fill={rgb, 255:red, 0; green, 0; blue, 0 }  ,fill opacity=1 ] (238,64) .. controls (238,61.24) and (240.24,59) .. (243,59) .. controls (245.76,59) and (248,61.24) .. (248,64) .. controls (248,66.76) and (245.76,69) .. (243,69) .. controls (240.24,69) and (238,66.76) .. (238,64) -- cycle ;
%Shape: Circle [id:dp7281946084111652] 
\draw  [fill={rgb, 255:red, 0; green, 0; blue, 0 }  ,fill opacity=1 ] (160,142) .. controls (160,139.24) and (162.24,137) .. (165,137) .. controls (167.76,137) and (170,139.24) .. (170,142) .. controls (170,144.76) and (167.76,147) .. (165,147) .. controls (162.24,147) and (160,144.76) .. (160,142) -- cycle ;
%Shape: Circle [id:dp8097386275656293] 
\draw  [fill={rgb, 255:red, 0; green, 0; blue, 0 }  ,fill opacity=1 ] (238,142) .. controls (238,139.24) and (240.24,137) .. (243,137) .. controls (245.76,137) and (248,139.24) .. (248,142) .. controls (248,144.76) and (245.76,147) .. (243,147) .. controls (240.24,147) and (238,144.76) .. (238,142) -- cycle ;
%Shape: Square [id:dp5604979304477404] 
\draw  [color={rgb, 255:red, 65; green, 117; blue, 5 }  ,draw opacity=1 ][line width=1.5]  (390,64) -- (468,64) -- (468,142) -- (390,142) -- cycle ;
%Straight Lines [id:da8071335207224892] 
\draw [color={rgb, 255:red, 65; green, 117; blue, 5 }  ,draw opacity=1 ][line width=1.5]    (468,64) -- (390,142) ;
%Shape: Circle [id:dp8534504887687471] 
\draw  [fill={rgb, 255:red, 0; green, 0; blue, 0 }  ,fill opacity=1 ] (385,142) .. controls (385,139.24) and (387.24,137) .. (390,137) .. controls (392.76,137) and (395,139.24) .. (395,142) .. controls (395,144.76) and (392.76,147) .. (390,147) .. controls (387.24,147) and (385,144.76) .. (385,142) -- cycle ;
%Shape: Circle [id:dp5250203321031937] 
\draw  [fill={rgb, 255:red, 0; green, 0; blue, 0 }  ,fill opacity=1 ] (463,142) .. controls (463,139.24) and (465.24,137) .. (468,137) .. controls (470.76,137) and (473,139.24) .. (473,142) .. controls (473,144.76) and (470.76,147) .. (468,147) .. controls (465.24,147) and (463,144.76) .. (463,142) -- cycle ;
%Straight Lines [id:da8013505185922052] 
\draw [color={rgb, 255:red, 208; green, 2; blue, 27 }  ,draw opacity=1 ][line width=1.5]    (390,64) -- (468,64) ;
%Shape: Circle [id:dp8649096026603069] 
\draw  [fill={rgb, 255:red, 0; green, 0; blue, 0 }  ,fill opacity=1 ] (385,64) .. controls (385,61.24) and (387.24,59) .. (390,59) .. controls (392.76,59) and (395,61.24) .. (395,64) .. controls (395,66.76) and (392.76,69) .. (390,69) .. controls (387.24,69) and (385,66.76) .. (385,64) -- cycle ;
%Shape: Circle [id:dp6174774830468734] 
\draw  [fill={rgb, 255:red, 0; green, 0; blue, 0 }  ,fill opacity=1 ] (463,64) .. controls (463,61.24) and (465.24,59) .. (468,59) .. controls (470.76,59) and (473,61.24) .. (473,64) .. controls (473,66.76) and (470.76,69) .. (468,69) .. controls (465.24,69) and (463,66.76) .. (463,64) -- cycle ;

% Text Node
\draw (109,179.4) node [anchor=north west][inner sep=0.75pt]    {$L_{G} =L_{|G|} =\begin{pmatrix}
2 & -1 & 0\\
-1 & 3 & -1\\
0 & -1 & 2
\end{pmatrix}$};
% Text Node
\draw (355,180.4) node [anchor=north west][inner sep=0.75pt]    {$L_{G_{\phi }} =\begin{pmatrix}
2 & \textcolor[rgb]{0.82,0.01,0.11}{+1} & 0\\
\textcolor[rgb]{0.82,0.01,0.11}{+1} & 3 & -1\\
0 & -1 & 2
\end{pmatrix}$};
% Text Node
\draw (156,42.4) node [anchor=north west][inner sep=0.75pt]    {$v_{1}$};
% Text Node
\draw (157,153.4) node [anchor=north west][inner sep=0.75pt]    {$v_{4}$};
% Text Node
\draw (235,153.4) node [anchor=north west][inner sep=0.75pt]    {$v_{3}$};
% Text Node
\draw (234,42.4) node [anchor=north west][inner sep=0.75pt]    {$v_{2}$};
% Text Node
\draw (207,112.4) node [anchor=north west][inner sep=0.75pt]    {$| G|$};
% Text Node
\draw (197,47.4) node [anchor=north west][inner sep=0.75pt]    {$+$};
% Text Node
\draw (197,81.4) node [anchor=north west][inner sep=0.75pt]    {$+$};
% Text Node
\draw (197,144.4) node [anchor=north west][inner sep=0.75pt]    {$+$};
% Text Node
\draw (147,94.4) node [anchor=north west][inner sep=0.75pt]    {$+$};
% Text Node
\draw (250,94.4) node [anchor=north west][inner sep=0.75pt]    {$+$};
% Text Node
\draw (381,42.4) node [anchor=north west][inner sep=0.75pt]    {$v_{1}$};
% Text Node
\draw (382,153.4) node [anchor=north west][inner sep=0.75pt]    {$v_{4}$};
% Text Node
\draw (460,153.4) node [anchor=north west][inner sep=0.75pt]    {$v_{3}$};
% Text Node
\draw (459,42.4) node [anchor=north west][inner sep=0.75pt]    {$v_{2}$};
% Text Node
\draw (435,112.4) node [anchor=north west][inner sep=0.75pt]    {$G_{\phi }$};
% Text Node
\draw (422,47.4) node [anchor=north west][inner sep=0.75pt]    {$-$};
% Text Node
\draw (421,81.4) node [anchor=north west][inner sep=0.75pt]    {$+$};
% Text Node
\draw (422,144.4) node [anchor=north west][inner sep=0.75pt]    {$+$};
% Text Node
\draw (372,94.4) node [anchor=north west][inner sep=0.75pt]    {$+$};
% Text Node
\draw (475,94.4) node [anchor=north west][inner sep=0.75pt]    {$+$};

\end{tikzpicture}
    \captionsetup{width=1.0\linewidth}
  \captionof{figure}{A Signed Graph and its Underlying Unsigned Graph}
  \label{fig:signedunsigned}

 \end{center}
\end{figure}

An important construction in the study of signed graphs is the notion of vertex switching.

\begin{defn}\label{def:switching}
Suppose $G_\phi$ is a signed graph, and let $v \in V$ be a vertex.  If we reverse the sign of every edge incident to $v$ we obtain a new function $\psi: E \rightarrow \{+,-\}$, and the resulting signed graph $G_\psi$ is said to be obtained by \emph{switching} the vertex $v$. A pair of signed graphs $G_{\phi}$ and $ G_{\phi'}$ (with the same underlying unsigned graph $G$) are \emph{switching equivalent} if there exists a sequence of vertex switchings on $G_\phi$ that results in $G_{\phi'}$.
\end{defn}

Allowing for sequences of vertex switchings defines an equivalence relation on the set of signed graphs with a fixed underlying graph $G$. From \cite[Lemma 3.1]{Zaslavsky} we have the following.

\begin{lem}\label{lem: switchingclass}
Suppose $G$ is a connected graph and let $T$ be a spanning tree. Each switching equivalence class of signed graphs on $G$ has a unique representative which is positive on $T$. In particular, the number of switching equivalence classes on $G$ is given by $2^{|E|-|V| +1}$.
\end{lem}

\begin{defn}
A signed graph $G_\phi$ is \emph{balanced} if it is switching equivalent to the all-positive graph $|G|$. 
\end{defn}

The well-known Kirchoff's Matrix-Tree Theorem says that the determinant of the reduced Laplacian matrix of a connected graph $G$ is given by the number of spanning trees of $G$ (see for instance \cite{Klivans} for connections to chip-firing). This has been generalized to the signed graph setting by Zaslavsky as follows.

\begin{thm}\cite[Theorem 8A.4]{Zaslavsky} \label{thm:Zaslavsky}
Suppose $G_{\phi}$ is a connected signed graph. Then the determinant of its (non-reduced) Laplacian $\tilde L$ is given by $$\det (\tilde L) = 
\sum_{\ell = 0}^{n}4^{\ell}b_{\ell},$$ where $b_{\ell}$ is the number of independent sets of edges that have $\ell$ cycles.
\end{thm}

Here a set of edges is \emph{independent} if it forms an independent set in the underlying matroid defined by the signed incidence matrix of $G_\phi$, we refer to \cite{Zaslavsky} for details. This result has been further generalized by Chaiken \cite{MR666857}, where one considers the determinant of the matrix obtained by removing a certain set of rows and columns from $\tilde L$. We note that $L$ is the special case of removing the row and column corresponding to the sink $q$.
In our work, we will often consider signed graphs where every edge not incident to the sink is negative (and where edges incident to the sink have arbitrary sign function).  
We say that $G_\phi$ is \emph{negative}, and write $G_\phi = -G$, if all edges not incident to the sink have a negative sign.  Note that in this case all entries of the reduced Laplacian $L=L_{G_\phi}$ are positive. This matrix is often referred to as the \emph{signless Laplacian} of $G$ in the literature \cite{FundBapat} .

In the case of negative signed graphs with a sink, Theorem \ref{thm:Zaslavsky} can be rephrased in simpler terms (see remarks after \cite[Theorem 8A.4]{Zaslavsky}). In what follows a \emph{$TU$-subgraph} is a subgraph of $G$ whose components are trees or unicylic subgraphs containing a single odd cycle. We then have the following result from \cite{FundBapat}.

\begin{thm}\cite[Theorem 7.8]{FundBapat})\label{thm:MatrixTreeNice}
Suppose $G_\phi$ is a negative signed graph with sink $q$, and let $L = L_{G_\phi}$ denote its reduced Laplacian.  Then the determinant of $L$ is given by
\[\det(L) = \sum_H 4^{c(H)},\]
\noindent
where the sum runs over all spanning $TU$-subgraphs $H$ of $G$ with $c(H)$ unicyclic components, and one tree component that contains the sink $q$.
\end{thm}

\subsection{Chip-firing and critical groups of signed graphs}

Now suppose $G_\phi$ is a signed graph with sink $q$ and nonsink vertices $v_1, \dots, v_n$. The reduced signed Laplacian $L = L_{G_\phi}$ naturally defines a chip-firing rule as discussed above: given an integer configuration $\vec{c} \in {\mathbb Z}^n$, firing the vertex $i$ leads to the configuration
\[\vec{c}^\prime = \vec{c} - L\vec{e}_i,\]
obtained by subtracting the $i$th column of $L$.  Hence the vertex $v_i$ loses $\deg(i)$ number of chips, passing a chip to each adjacent vertex which shares a positive edge, and deleting a chip at each adjacent vertex which shares a negative edge. 

\begin{figure}[h]
 \begin{center}
 
 % \vspace{3.7mm}
    \input{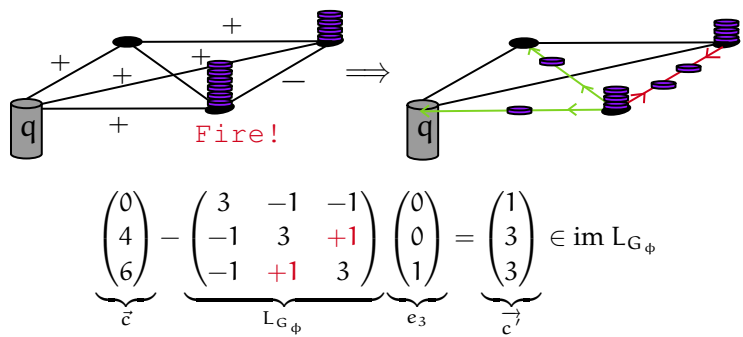}
    \captionsetup{width=1.0\linewidth}
  \captionof{figure}{Firing a Vertex}
  \label{fig:examplefire}

 \end{center}
\end{figure}

The matrix $L = L_{G_\phi}$ also defines an equivalence relation on integer vectors.  Namely for $\vec{c}_1, \vec{c}_2 \in {\mathbb Z}^n$ we say that
\[\vec{c}_1 \sim_L \vec{c}_2 \text{ if and only if } \vec{c}_2 = \vec{c}_1 - L\vec{z}, \text{ for some $\vec{z} \in {\mathbb Z}^n$}.\]
\noindent
%In other words, $\vec{c}_1$ and $\vec{c}_2$ are equivalent if there exists some sequence of fires (or reverse-fires) which takes $\vec{c}_1$ to $\vec{c}_2$, without requiring nonnegativity of any entry.
The equivalence relation defined by $L$ leads to the following notion of a critical group of $G_\phi$.

\begin{defn}
Suppose $G_\phi$ is a signed graph with nonsink vertices $v_1, \dots, v_n$, and let $L = L_{G_\phi}$ denote its reduced signed Laplacian. The \emph{critical group} of $G_\phi$ is the abelian group
\[{\mathcal K}(G_\phi) = {\mathbb Z}^n/\im  L.\]
\end{defn}

If the underlying graph $G$ is connected, one can see that the matrix $L$ has rank $n$ by applying \cite[Theorem 8A.4]{Zaslavsky}. From this it follows that ${\mathcal K}(G_\phi)$ is a finite group of order $|\mathcal{K}(G_\phi)| = \det(L)$.

\begin{rem}\label{snf}
To compute the critical group of a signed graph $G_\phi$ with chosen sink $q$, one can perform integral row and column operations on the matrix $L_{G_\phi}$ to obtain its \emph{Smith normal form}:  
$$\begin{pmatrix} p_1 & 0 & 0 & \dots & 0\\
0 & p_2 & 0 & \dots & 0\\
0 & 0 & p_3 & \dots & 0\\
\vdots & \vdots & \vdots & \ddots & 0\\
0 & 0 & 0 & \dots & p_{n} \end{pmatrix}.$$
In this case the diagonal elements satisfy $p_i | p_{i+1}$ for all $i = 1, \dots, n-1$ and are called the \emph{invariant factors}. It then follows that
\[\mathcal{K}(G_\phi) \cong \Z/p_1\Z \oplus ... \oplus \Z/p_n\Z.\]
\end{rem}

\begin{rem}
Note that if $G$ has sink $q$ and $G \backslash \{q\}$ has more than one component, then the reduced Laplacian $L$ has a block matrix form.  The row operations that reduce $L$ to Smith Normal Form can be chosen to preserve this property, implying that the critical group ${\mathcal K}(G)$ is a direct sum of the critical groups of the underlying subgraphs.  Hence in our computations of ${\mathcal K}(G)$ we can assume that $G \backslash \{q\}$ is connected.
\end{rem}

By the Matrix Tree Theorem for signed graphs \ref{thm:Zaslavsky}, the order of ${\mathcal K}(G_\phi)$ also has a combinatorial interpretation. We note that ${\mathcal K}(G_\phi)$ (and even its order, given by $\det(L)$) depends on the choice of sink vertex $q$.

Our goal in this paper is to study the dynamics of chip-firing and the structure of the critical group of a signed graph $G_\phi$ in terms of various notions of stability, in analogy with the theory for graphs.  One could study the dynamics of chip-firing defined by $L_{G_\phi}$ using nonnegative integral vectors as valid configurations, but it is not hard to see that many desired properties break down. In particular, one does not have a natural representative for each element of ${\mathcal K}(G_\phi)$, as the next example illustrates.

\begin{ex}
Let $-C_3$ denote the 3-cycle with a single negative edge incident to the nonsink vertices. First note that  $$L_{-C_3}=\begin{pmatrix}2&1\\1&2\end{pmatrix}$$ and so $|{\mathcal K}(-C_3)| = 3$. If we allow all nonnegative vectors as valid configurations, we obtain an infinite collection of critical configurations in a single element of ${\mathcal K}(-C_3)$. To see this first note that all configurations of the form $(n,0)^T$ or $(0,n)^T$, where $n \in \mathbb{Z}_{\geq 0},$ are both critical and $z$-superstable. One can then check that all vectors of the form $(1 + 3k, 0)^T$ are in the same equivalence class determined by $L_{-C_3}$.
\end{ex}

Hence the naive approach of using nonnegative configurations to define chip-firing on signed graphs fails. However, signed graphs provide a natural setting to employ a theory of chip-firing on general invertible matrices developed in \cite{GuzKlivans}.  We review the basics of this theory next.

\subsection{Chip-firing pairs} \label{sec:chipfiring pairs}
In \cite{GuzKlivans}, Guzm\'an and Klivans develop a theory of chip-firing where the chip-firing rule is governed by an arbitrary invertible matrix.   Here one fixes a pair $(L,M)$ of $n \times n$ matrices, where $L$ is invertible and $M$ is an \emph{$M$-matrix}. To recall this latter concept, we have the following characterization from \cite{gabrielov, Plemmons}.

\begin{defn}\label{def:Mmatrix}
Suppose $M$ is an $n \times n$ matrix such that $(M)_{ii} > 0$ for all $i$ and $(M)_{ij} \leq 0$ for all $i \neq j$. Then $M$ is called an (invertible) \emph{$M$-matrix} if any of the following equivalent conditions hold:
\begin{enumerate}
    \item $M$ is avalanche finite;
    \item The real part of the eigenvalues of $M$ are positive;
    \item The entries of $M^{-1}$ are non-negative;
    \item There exists a vector $\vec{x} \in {\mathbb R}^{n}$ with ${\vec x} \geq \vec{0}$ such that $M \vec{x}$ has all positive entries.
\end{enumerate}
\end{defn}

The chip-firing model is determined by the choice of the pair $(L,M)$, which we call a \emph{chip-firing pair}. The relevant (chip) \emph{configurations} $\vec{c} \in \mathbb{Z}^{n}$ are simply integer vectors with $n$ entries, and chip-firing is dictated by the matrix $L$.  In this context the notion of a \emph{valid} configuration takes on a new meaning.

\begin{defn}\label{defn:valid}
Suppose $(L,M)$ is a chip-firing pair.  A configuration $\vec{c}$ is \emph{valid} if $\vec{c} \in S^+$, where
\[S^+ = \{LM^{-1} \vec{x} : LM^{-1}\vec{x} \in \mathbb{Z}^{n}, \vec{x} \in \mathbb{R}^n_{\geq 0}\}. \]

Equivalently, a configuration $\vec{c}$ is valid if $ML^{-1}\vec{c} \in R^+$, where
\[R^+ = \{ \vec{x} \in \mathbb{R}^n_{\geq 0} : LM^{-1}\vec{x} \in \mathbb{Z}^{n} \}.\]
\end{defn}

\begin{rem}
Note that if $M$ is an $n \times n$ invertible $M$-matrix, then one can take $(M,M)$ as a chip-firing pair. In this case $M$ defines the chip-firing rule and the set of valid configurations is given by $S^+ = \mathbb{Z}^n_{\geq 0}$, integer vectors with nonnegative entries.  This recovers the notion of chip-firing on $M$ as studied in \cite{GuzKliMmatrices}.
\end{rem}

Next we collect some other relevant definitions  from \cite{GuzKlivans}.

\begin{defn}
Suppose $(L,M)$ is a chip-firing pair, and suppose that $\vec{c} \in S^+$ is a valid configuration. A site $i \in \{1,\dots, n\}$ is \emph{ready to fire} if
\[\vec{c} - L\vec{e}_i \in S^+,\]
so that the vector obtained by subtracting the $i$th row of $L$ from $\vec{c}$ is also valid.

Similarly, suppose $\vec{x} \in R^+$. Then a site $i \in \{1,\dots, n\}$ is \emph{ready to fire} if 
\[\vec{x} - M\vec{e}_i \in R^+.\]

A configuration $\vec{c}$ (in $S^+$ or $R^+$) is \emph{stable} if no site is ready to fire.
\end{defn}

If $i$ is ready to fire, we say that $\vec{b} = \vec{c} - L \vec{e}_i \in S^+$ is obtained from $\vec{c}$ by a \emph{legal firing}. Iterating this 
process, we say that $\vec{a} \in S^+$ is obtained from $\vec{c}$ via a \emph{sequence of legal firings}. For a configuration $\vec{c} \in S^+$ (resp. $\vec{d} \in R^+$), we let $\text{stab}_{S^+}(\vec{c})$ (resp. $\text{stab}_{R^+}(\vec{d})$) denote the configuration obtained by performing a sequence of legal firings until no site is ready to fire. By adapting the proof of \cite[Theorem 2.2.2]{Klivans} one can see that both $\text{stab}_{S^+}(\vec{c})$ and $\text{stab}_{R^+}(\vec{d})$ are unique.
If the context is clear we simply write $\stab(\vec{x})$ for the stabilization of the configuration $\vec{x}$. 

Analogous to the superstable configurations of usual chip-firing we also need the notion of (multi) set-firing.

\begin{defn}\label{def:superstable}
Suppose $(L,M)$ is a chip-firing pair. A configuration $\vec{c} \in S^+$ is \emph{$z$-superstable} if $$\vec{c}  - L\vec{z} \notin S^+$$ for all nonzero $\vec{z} \in {\mathbb Z}^n_{\geq 0}$.

Similarly, a valid configuration $\vec{x} \in R^+$ is $z$-superstable if $$\vec{x} - M\vec{z} \notin R^+$$ for all nonzero $\vec{z} \in {\mathbb Z}^n_{\geq 0}$.

A configuration $\vec{c} \in S^+$ (resp. $R^+$) is \emph{$\chi$-superstable} if $$\vec{c}  - L\vec{\chi} \notin S^+ \text{(resp. $R^+$)}$$  for all nonzero $\vec{\chi} \in \{0,1\}^n$.
\end{defn}

\begin{defn}
Suppose $(L,M)$ is a chip-firing pair. A configuration $\vec{c} \in S^+$ is said to be \emph{reachable} if there exists a configuration $\vec{d} \in S^+$ satisfying 
\begin{enumerate}
    \item $\vec{d} - L\vec{e}_i \in S^+$ for all $i = 1, \dots, n$, and
    \item $\vec{c} = \vec{d} - \sum_{j=1}^k L \vec{e}_{i_j}$, where $\vec{d} - \sum_{j=1}^\ell L\vec{e}_{i_j} \in S^+$ for all $\ell \leq k$.
\end{enumerate}
Similarly, a configuration $\vec{x} \in R^+$ is reachable if the analogous conditions hold for some $\vec{y}$ in $R^+$, where $M$ replaces the role of $L$.
\end{defn}

Hence $\vec{c}$ is reachable if it can be obtained from a sequence of legal firings starting from a valid configuration $\vec{d}$, where any vertex of $\vec{d}$ can be fired. With these preliminaries we can recall the notion of a critical configuration.

\begin{defn} \label{def:critical}
Suppose $(L,M)$ is a chip-firing pair.  A configuration $\vec{c} \in S^+$ is said to be \emph{critical} if it is stable and \emph{reachable}.
Similarly,  a configuration $\vec{c} \in R^+$ is said to be \emph{critical} if it is stable and reachable.
\end{defn}

As was the case for classical chip-firing, the matrix $L$ defines an equivalence relation on integer vectors.  Namely, for $\vec{c}, \vec{d} \in {\mathbb Z}^n$ we have $\vec{c} \sim_{L} \vec{d}$ if there exists a vector $\vec{f} \in {\mathbb Z}^n$ such that $\vec{d} =  \vec{c} - L{\vec f}$. Similarly if $\vec{x}, \vec{y} \in {\mathbb R}^n$ we have $\vec{x} \sim_{M} \vec{y}$ if there exists $\vec{z} \in {\mathbb Z}^n$ such that $\vec{y} = \vec{x} - M{\vec z}$. As noted in \cite{GuzKlivans}, one can freely move between $S^+$ and $R^+$ when checking for for equivalence and hence other properties. We record these results here.

\begin{lem}\label{lem:GuzKlivansSR}
Suppose $(L, M)$ is a chip-firing pair. Let $\vec{c} = LM^{-1}\vec{x}$, where $\vec{x} \in R^+$ (so that $\vec{c} \in S^+$). Then for any $\vec{d} = LM^{-1}\vec{y}$ with $\vec{y} \in R^+$, we have $\vec{c} \sim_{L} \vec{d}$ if and only if $\vec{x} \sim_{M} \vec{y}$. From this we conclude the following:
\begin{enumerate}
    \item $\vec{c}$ is stable if and only if $\vec{x}$ is stable.
    \item $\vec{c}$ is reachable if and only if $\vec{x}$ is reachable.
    \item $\vec{c}$ is critical if and only if  $\vec{x}$ is critical.
    \item $\vec{c}$ is $z$-superstable if and only if $\vec{x}$ is $z$-superstable.
    
\end{enumerate}

\end{lem}

\begin{lem}\label{lem:FiringSeq}
% Suppose $G_\phi$ is a signed graph with chip-firing pair $(L,M)$, and let $\vec{b}$ and $ \vec{c}$ be valid configurations.
Consider chip-firing pair $(L,M)$, and let $\vec{b}$ and $ \vec{c}$ be valid configurations. $\stab_{S^+}(\vec{b}) = \vec{c}$ if and only if $\stab_{R^+}(ML^{-1}\vec{b}) = ML^{-1}\vec{c}$.
\end{lem}
\begin{proof}
This follows from the fact that $\vec{c}$ is obtained by a legal firing sequence from $\vec{b}$ on $G_\phi$ (there exists some $v$ such that $\vec{b} - L\vec{e_v} = \vec{c}$) if and only if $ML^{-1}\vec{c}$ is obtained by a legal firing from $ML^{-1}\vec{b}$ on $|G|$ (there exists some $v$ such that $ML^{-1}\vec{b} - M \vec{e_v} = ML^{-1} \vec{c}$).
\end{proof}

 Recall that the equivalence classes $[\vec{c}]_L$ defined by $\sim_{L}$ form the elements of the critical group of the chip firing pair to be ${\mathbb Z}^n/\im L$. A main result from \cite{GuzKlivans} is the following.

\begin{thm}\cite[Theorems 3.5, 4.3, 5.5]{GuzKlivans}\label{thm:class}
Suppose $(L,M)$ is a chip-firing pair.  Then there exists exactly one z-superstable configuration and one critical configuration in each equivalence class $[\vec{c}]_L$.
\end{thm}

\subsection{Our setup}
We now fix a signed graph $G_\phi$ with sink vertex $q$ and nonsink vertices $v_1, \dots, v_n$, and let $G = |G|$ denote the underlying graph. Recall that the reduced Laplacian $L = L_{G_\phi}$ naturally defines a chip-firing rule, as well as a critical group ${\mathcal K}(G_\phi) := {\mathbb Z}^n/\im L$.  Since the reduced Laplacian $L_G$  is an $M$-matrix we can employ the Guzm\'an-Klivans theory.

\begin{defn}
Suppose $G_\phi$ is a signed graph with sink vertex $q$ and underlying graph $G$. Then $(L_{G_\phi}, L_G)$ forms a \emph{chip-firing pair} for $G_\phi$.
\end{defn}

From the Guzm\'an-Klivans theory we obtain the notion of \emph{critical} and \emph{$z$-superstable} configurations of a signed graph.  Each equivalence class $[\vec{c}]_L$ contains a unique valid critical and a unique valid $z$-superstable.

Furthermore, we can recover the critical group $\mathcal{K}(G)$ in terms of the critical configurations of $G$. In particular one can check that the stablization of the sum of two critical configurations is again critical, so that if $\vec{v}_1, \vec{v}_2$ are critical configuration we have 
$$[\vec{v}_1] \oplus [\vec{v}_2] = [\text{stab}(\vec{v}_1+\vec{v}_2)],$$
provides a critical representative of the product.

\subsection{Relation to other critical groups}
\label{sec: other}
There is another way to define a critical group of a signed graph, for instance studied by Reiner and Tseng in \cite{Reiner} based on work of Bacher et al \cite{Bacher}. We discuss the main ideas here and refer to \cite{Bacher} and \cite{Reiner} for details.  The definition is inspired by the following construction of a critical group for any $r$-dimensional rational subspace $\Lambda^{\mathbb R} \subset {\mathbb R}^m$.  Let $\Lambda = \Lambda^{\mathbb R} \cap {\mathbb Z}^m$ denote the underlying lattice and let $\Lambda^\#$ denote its \emph{dual lattice}
\[\Lambda^\# := \{ x \in \Lambda^{\mathbb R}: \langle x, \lambda \rangle \in {\mathbb Z} \text{ for all $\lambda \in \Lambda$}\}.\]

The quotient $\Lambda^\#/\Lambda$ is called the \emph{determinant group} of $\Lambda$. Given any orthogonal decomposition of Euclidean space ${\mathbb R}^m = B^{\mathbb R} \oplus Z^{\mathbb R}$ we obtain two determinant groups which turn out be isomorphic to the group
\[\mathcal{C}_{BZ} := {\mathbb Z}^m/(B \oplus Z).\]

To compute ${\mathcal C}_{BZ}$ we can use either an integral basis $\{b_1, \dots, b_r\}$ of $B$ or an integral basis $\{z_1, \dots, z_s\}$ of $Z$. For this let $M_B$ and $M_Z$ be the matrices with columns given by those vectors. In what follows we use $G(A):= A^TA$ to denote the \emph{Gram matrix} of a matrix $A$.  We then get 
\[\mathcal{C}_{BZ} \cong \text{coker}(G(M_Z)) \cong \text{coker}(G(M_B)).\]

The canonical example of this construction comes from the decomposition of ${\mathbb R}^E$ given by the incidence matrix $\delta$ of a graph $G = (V,E)$.  The row space of $\delta$ is called the \emph{cut space}, and the null space (which is always orthogonal to the row space) is called the \emph{cycle space}.  Since the integral cut space of any connected graph has an integral basis given by the vertex cuts of any subset of $n-1$ vertices, this construction recovers the usual critical group ${\mathcal K}(G)$. In particular the relevant Gram matrix is simply the reduced Laplacian $L_G$ (for some choice of sink vertex).  One can also take the integral \emph{cycle} space of $G$ to recover ${\mathcal K}(G)$, as was studied in \cite{DMSR}.

In \cite{Reiner} the authors use this approach to define a critical group ${\mathcal C}(G_\phi)$ of a signed graph $G_\phi$.  We let $\delta$ denote the signed incidence matrix of $G_\phi$ and define
\[{\mathcal C}(G_\phi) = \im \delta/\im \delta \delta^T = \im \delta/\im \tilde{L},\]
\noindent
where $\tilde{L} = \tilde{L}_{G_\phi}$ is the signed (but non-reduced) Laplacian of $G_\phi$.

If $G_\phi$ is connected and has all positive edges, one can check that $\delta$ has Smith normal with diagonal entries $(1,1, \dots, 1,0)$, and similarly the Smith normal form of $\tilde{L}$ has a $0$ in the bottom right.  Hence one can compute ${\mathcal C}(G_\phi)$ by reading off the nonzero entries in the Smith normal form of $\tilde{L}$.  The critical group ${\mathcal C}(G_\phi)$ has pleasing functorial properties and also connects to the \emph{double cover} of graphs associated to $G_\phi$. We refer to \cite{Reiner} for details.

The critical group ${\mathcal K}(G_\phi)$ that we study in this paper more closely connects to the notion of `adding critical configurations', and also has the property that its order $|{\mathcal K}(G)|$ is given by $\det L_{G_\phi}$.  To see how the two constructions differ, consider the signed graph $G_\phi$ given by a $3$-cycle with a single negative edge between the nonsink vertices. In this case the reduced signed Laplacian $ L = L_{G_\phi}$ is given by
\[L = \begin{pmatrix} 2 & 1 \\ 1 & 2 \end{pmatrix}.\]

Note that $\det(L) = 3$ so we expect $|{\mathcal K}(G_\phi)| = 3$. Indeed one can check (see Section \ref{sec:cycle}) that $\mathcal{K}(G_\phi) \cong {\mathbb Z}/3{\mathbb Z}$.  On the other hand, a simple computation shows that $\mathcal{C}(G_\phi) \cong {\mathbb Z}/2 {\mathbb Z}$ .

\section{Proofs}\label{sec:generalizedResults}

We next turn to proofs of our main results. For our first collection of statements, we consider a general integral chip-firing pair $(L,M)$ not necessarily arising from a signed graph. In what follows, for a vector $\vec{v} \in {\mathbb R}^n$, we let $\lfloor \vec{v} \rfloor$ denote the vector obtained by taking the floor of each entry, so that $(\lfloor \vec{v} \rfloor)_i = \lfloor (\vec{v})_i \rfloor$.

\begin{prop}\label{prop:keyFund}
Suppose $(L,M)$ is an integral chip-firing pair, where $L$ is an invertible matrix and $M$ is an $M$-matrix. Suppose $\vec{c} \in S^+$ is a valid configuration, and let $\vec{b} \in \Z^n_{\geq 0}$ . Then $\vec{c} - L\vec{b} \in S^{+}$ if and only if $\lfloor ML^{-1} \vec{c} \rfloor - M \vec{b} \in \R^n_{\geq 0}$. 
\end{prop}
\begin{proof}
First assume that $\lfloor ML^{-1} \vec{c} \rfloor - M \vec{b} \in \R^n_{\geq 0}$. It then follows that $ML^{-1}\vec{c} - M \vec{b}$ has nonnegative entries.
Applying $LM^{-1}$ to this vector results in $\vec{c} - L \vec{b}$, which has integral entries since we are assuming that $L$ is integral. We conclude that $\vec{c} - L \vec{b} \in S^+$, as desired.

On the other hand, suppose $\vec{c} - L\vec{b} \in S^+$. Then we have that $ML^{-1}\vec{c} -M\vec{b} \in R^+$ by definition of $R^+.$ Therefore $\lfloor ML^{-1}\vec{c} \rfloor -M\vec{b}  = \lfloor ML^{-1}\vec{c} -M\vec{b}  \rfloor \in \mathbb{R}^n_{\geq 0}$. This completes the proof.
\end{proof}

\begin{thm}\label{thm:KeyIdea}
Suppose $(L,M)$ is a integral chip-firing pair, where $L$ is an invertible matrix and $M$ is an $M$-matrix. Suppose $\vec{c} \in S^+$ is a valid configuration. Then $\vec{c}$ is $z$-superstable if and only if $\lfloor ML^{-1} \vec{c} \rfloor$ is $z$-superstable for $M$. Similarly $\vec{c}$ is critical if and only if $\lfloor ML^{-1} \vec{c} \rfloor$ is critical for $M$.
\end{thm}
\begin{proof}
Recall that $\vec{c}$ is $z$-superstable if and only if there is no multiset firing that gives rise to a valid configuration, which by Proposition \ref{prop:keyFund} is equivalent to the configuration $\lfloor ML^{-1}\vec{c} \rfloor$ having no multiset firing that gives rise to a nonnegative configuration. This proves the first statement.

We proceed to the second statement of the theorem. 
%By Lemma \ref{lem:GuzKlivansSR}, it suffices to show that $ML^{-1}\vec{c}$ is critical (as a configuration in $R^+$) if and only if $\lfloor ML^{-1}\vec{c} \rfloor$ is critical for $M$. 
Proposition~\ref{prop:keyFund} tells us that a sequence of site firings starting with $\vec{c}$ is a legal firing sequence (as a sequence of configurations in $S^+$ determined by the chip-firing pair $(L,M)$) if and only if it is a legal firing sequences beginning with $\lfloor ML^{-1}\vec{c} \rfloor$ (as a sequence of configurations in ${\mathbb Z}^n_{\geq 0}$ for $M$). Hence $\vec{c}$ is reachable for the pair $(L,M)$ if and only if $\lfloor ML^{-1}\vec{c} \rfloor$ is reachable for $M$, and $\vec{c}$ is stable for $(L,M)$ if and only if $\lfloor ML^{-1}\vec{c} \rfloor$ is stable for $M$. This proves the claim about criticality.
\end{proof}

\subsection{Special configurations on signed graphs}
Recall that if $G_\phi$ is a signed graph with sink $q$ and nonsink vertices $\{v_1, \dots, v_n\}$, we define the chip-firing pair $(L,M)$ where $L=L_{G_{\phi}}$ and $M=L_{G}$. As discussed above, this defines the notion of a valid configuration, as well as $z$-superstable and critical configurations. We first prove Proposition \ref{prop:positive}, which provides a restriction on the set of valid configurations. This result also appears in \cite{Ram}, we provide a proof for completeness.

\begin{prop}\cite[Proposition 1]{Ram}\label{prop:positive}
    For any signed graph $G_\phi$ with chosen sink $q$, we have $S^+ \subset {\mathbb Z}^n_{\geq 0}$, so that the set of valid configurations determined by the chip-firing pair $(L,M)$ is contained in the nonnegative orthant.
\end{prop}

\begin{proof}
Recall that the columns of $LM^{-1}$ define the cone of the valid configurations. 
    Since $L$ is obtained from $M$ by changing some non-diagonal $-1$ entries to $+1$, we can write $L=M+N,$ where $$(N)_{i,j}=\begin{cases} 2, & \text{ if }L_{i,j}=+1 \\ 0, & \text{ otherwise. }\end{cases}.$$ Hence $LM^{-1}=(M+N)M^{-1}=MM^{-1}+NM^{-1}.$ Since $M$ is an $M$-matrix, we have that $M^{-1}$ has nonnegative entries.  The same is true for $MM^{-1} = I$ as well as for $N$, so that $LM^{-1}$ has all nonnegative entries and the claim follows.
        %Since $MM^{-1}$ has all nonnegative entries and $M^{-1}$ has nonnegative entires, this implies that $LM^{-1}$ has all nonnegative entries, and the claim follows.
\end{proof}

We also have the following important observation.

\begin{cor}\label{cor:KeyIdea2}
Let $G_\phi$ be a signed graph with sink $q$ and underlying graph $G$. Suppose $\vec{c} \in S^+$ is a valid configuration. Then $\vec{c}$ is $z$-superstable if and only if $\lfloor ML^{-1} \vec{c} \rfloor$ is superstable for $G$. Similarly $\vec{c}$ is critical if and only if $\lfloor ML^{-1} \vec{c} \rfloor$ is critical for $G$. 
\end{cor}

\begin{proof}
   Suppose $\vec{c}$ is a valid configuration for the signed graph $G_\phi$. For one direction, if $\vec{c}$ is $z$-superstable then from Theorem \ref{thm:KeyIdea} we know that $\lfloor \vec{c} \rfloor$ is $z$-superstable as a configuration for $G$, and hence superstable. For the other direction, if  $\lfloor \vec{c} \rfloor$ is a superstable configuration for the graph $G$ then from \cite[Theorem 4.4]{GuzKlivans} we know that  $\lfloor \vec{c} \rfloor$ is $z$-superstable, since these notions coincide for graphs. Hence from Theorem \ref{thm:KeyIdea} we conclude that $\vec{c}$ is $z$-superstable. The statement for critical configurations directly follows from Theorem \ref{thm:KeyIdea}.
   \end{proof}

From this result, one can conclude that the notion of $z$-superstability and $\chi$-superstablity also coincide for signed graphs. In particular when checking for $z$-superstability, one can restrict to set (as opposed to multiset) firings. Here we include a self-contained proof, in part to clarify some notational ambiguity in \cite{GuzKlivans}.

\begin{thm}\label{thm:zchigeneral}
Suppose $G_{\phi}$ is a signed graph with sink $q$ and chip-firing pair $(L,M)$.  Then a valid configuration $\vec{c} \in S^+$ is $z$-superstable if and only if $ML^{-1}\vec{c} - M\vec{z} \not\geq \vec{0}$ for any nonzero vector $\vec{z}$ in $\{0, 1\}^{n}$. 
\end{thm}

\begin{proof}
Here we follow the strategy employed in the proof of \cite[Theorem 4.4]{GuzKliMmatrices}.  Note, however, that there is an error in the notation used in \cite{GuzKliMmatrices} which we seek to clarify here.  From the definition, we know that if $\vec{c}$ is a $z$-superstable configuration then it is also $\chi$-superstable, and hence we only need to prove the converse.

For this suppose $\vec{c} \in S^+$ is $\chi$-superstable, so that $ML^{-1}\vec{c} - M\vec{\chi} \not\geq \vec{0}$ for any vector 0-1 vector $\vec{\chi}$.  To establish the result it suffices to show that $ML^{-1}\vec{c} - M\vec{z} \not\geq \vec{0}$ for any nonzero, nonnegative integer vector $\vec{z}$. To that end, fix a nonzero vector $\vec{z} \in \mathbb{N}^n$. For brevity let $\vec{x} = ML^{-1}\vec{c}$.

Now we construct a 0-1 vector $\vec{\chi}$ as follows. Let $\max (\vec{z})$ denote the maximum value of the entries of $z$, and for each $j$ such that $(\vec{z})_j = \max (\vec{z})$, we set $(\vec{\chi})_j = 1$ and $(\vec{\chi})_j = 0$ otherwise. By $\chi$-superstability of $\vec{x} \in R^+$, there exists $j$ such that $(\vec{x} - M\vec{\chi})_j < 0$. From $(M\vec{\chi})_j > 0$ and the definition of $M$, $(\vec{\chi})_j = 1$.  Note that this construction of $\vec{\chi}$ makes it so that $(\vec{z} - \vec{\chi})\geq \vec{0}$ and $(\vec{z} - \vec{\chi})_j = \max (\vec{z} - \vec{\chi})$ for any $j$ such $(\vec{\chi})_j = 1$.

For $j$ such that $(\vec{x} - M\vec{\chi})_j<0$, we claim that  $(\vec{x} - M\vec{z})_j \leq (\vec{x} - M\vec{\chi})_j < 0$. To see this it suffices to show $(M(\vec{z} - \vec{\chi}))_j \geq 0$. By applying the fact that $(\vec{\chi})_j = 1$ and by the non-negative row-sum of $M$, we find that \begin{align*}(M(\vec{z} - \vec{\chi}))_j &= \mbox{deg}(v_j)\cdot (\vec{z} - \vec{\chi})_j - \sum_{k: \mbox{ $\{v_k, v_j\}$ is an edge}}(\vec{z} - \vec{\chi})_k\\&= \mbox{deg}(v_j)\cdot \max (\vec{z} - \vec{\chi}) - \sum_{k: \mbox{ $\{v_k, v_j\}$ is an edge}}(\vec{z} - \vec{\chi})_k\\& \geq (\max (\vec{z} - \vec{\chi})) \cdot \sum_{i} M_{j,i} \\& \geq 0\end{align*}
This proves the claim. We conclude that $\vec{x} \in R^+$, and hence $\vec{c} = LM^{-1}\vec{x} \in S^+$, is $z$-superstable.
\end{proof}

\begin{rem}
Note that the result of Theorem \ref{thm:zchigeneral} relies only on the fact that the matrix $M = L_G$ has nonnegative row sums, as was the case in \cite[Theorem 4.4]{GuzKliMmatrices}.
\end{rem}

\subsection{Checking for and constructing special configurations}\label{sec:check}

Suppose $G_\phi$ is a signed graph and $\vec{c} \in {\mathbb Z}^n$ is any integer vector. From the definitions, it is not clear how one could determine if $\vec{c}$ was $z$-superstable, since a priori one would have to determine all possible multiset firings. However, our results provide an efficient way to check this condition. 

First, to determine if $\vec{c}$ valid, we compute $ML^{-1}\vec{c}$ and check if all entries are nonnegative.
Having verified that it is valid, now to check for $z$-superstability, we check whether $\lfloor ML^{-1}\vec{c} \rfloor$ is a superstable configuration of $G$: we can do this via running Dhar's \emph{burning algorithm} \cite{PhysRevLett.64.1613}, which provides a polynomial time (in the number of vertices) way to check this condition.

Note that a similar approach allows us to check if a given vector $\vec{c} \in S^+$ is critical for $G_\phi$. For this we use the fact that the set of critical configurations of the unsigned graph $G$ can be obtained from the set of superstable configurations for $G$.

We can also use our results to \emph{produce} the set of all $z$-superstable (and critical) configurations for a given signed graph $G_\phi$. This improves the somewhat ad hoc algorithm for generating these configurations as discussed in \cite{Ram}.

\begin{prop}\label{prop:algorithm}
For any signed graph $G_\phi$, there exists an algorithm to construct the set of $z$-superstable configurations and the set of critical configurations.
\end{prop}

\begin{proof}
Given a signed graph $G_\phi$, we first produce the set of ${\mathcal S}(G)$ superstable configurations for the underlying graph $G$.
One method for this is to first construct the $G$-parking function ideal $M_G$ (see \cite{PostnikovShapiro03} for details), the exponent vectors of the standard monomials of $M_G$ produce the superstable configurations. One then generates the set ${\mathcal C}(G)$ of critical configurations of $G$ by taking the set of dual configurations:
\[{\mathcal C}(G) = \{\vec{c}_{\max} - \vec{c}: \vec{c} \in {\mathcal S}(G)\}.\]
Recall that $\vec{c}_{\max}$ is the `maximal superstable' configuration for $G$ defined by $(\vec{c}_{\max})_i = \deg(v_i) -1$.

Now, for each $\vec{c} \in {\mathcal S}(G)$ we consider all vectors in the semiclosed hypercube $[c_1, c_1 + 1) \times [c_2, c_2 + 1) \times ... \times [c_n, c_n + 1)$ that lie in the lattice determined by $\frac{1}{\det L}{\mathbb Z}^n$. For each such vector $\vec{x}$, we compute $LM^{-1}\vec{x}$ to determine if it lies in $S^+$ (i.e. has integer coordinates). The configuration $\vec{x}$  is superstable if and only if $LM^{1}\vec{x}$ has integer coordinates. 

Note that vectors in $R^+$ that we must consider are of the form $ML^{-1}\vec{x}$, where $\vec{x}$ has integer entries. Recall that $L^{-1}$ can be computed as $L^{-1} = \frac{1}{\det L} C$, where $C$ has integer entries. Hence vectors in the lattice $\frac{1}{\det L}{\mathbb Z}^n$ will be sufficient to produce the desired configurations.

For the critical configurations, we consider a similar strategy for every vector $\vec{d} \in {\mathcal C}(G)$.
\end{proof}

\begin{ex}
To illustrate our algorithm, we let $G_\phi$ be the signed graph given by the 3-cycle with a single negative edge not incident to the sink. Here the relevant matrices are 
\[ L = \begin{pmatrix} 2 & 1 \\ 1 & 2 \end{pmatrix}, \quad \quad M = \begin{pmatrix} 2 & -1 \\ -1 & 2 \end{pmatrix}, \quad \quad LM^{-1} = \begin{pmatrix} \frac{5}{3} & \frac{4}{3} \\ \frac{4}{3} & \frac{5}{3} \end{pmatrix}.\]

The $G$-parking function ideal for this case is given by $\langle x_1^2, x_2^2, x_1x_2\rangle$, and hence the superstable configurations for $G$ are ${\mathcal S}(G) = \left\{\ctwovector{0}{0}, \ctwovector{1}{0}, \ctwovector{0}{1} \right\}$. Here we have $\vec{c}_{\max} = \ctwovector{1}{1}$ and hence the critical configurations for $G$ are ${\mathcal C}(G) = \left\{ \ctwovector{1}{1}, \ctwovector{0}{1}, \ctwovector{1}{0}\right\}$. 

Next we construct the $z$-superstable congfiguration for $G_\phi$. Since $\det L = 3$ we consider all vectors of the form $\ctwovector{\frac{i}{3}}{\frac{j}{3}}$, $\ctwovector{1 + \frac{i}{3}}{\frac{j}{3}}$, $\ctwovector{\frac{i}{3}}{1 + \frac{j}{3}}$, where $i, j = 0,1,2$. For each of these vectors we multiply on the left by $LM^{-1}$ and check integrality.
For example $LM^{-1}\ctwovector{0}{\frac{1}{3}} = \ctwovector{\frac{4}{9}}{\frac{5}{9}}$, but $LM^{-1}\ctwovector{\frac{1}{3}}{\frac{1}{3}} = 
\ctwovector{1}{1}$, so that $\ctwovector{1}{1}$ is a $z$-superstable configuration for $G_\phi$.
The story for critical configurations is similar.
\end{ex}

    In the context of chip-firing on a graph $G$, the set of superstable configurations is closed under coordinate-wise $\leq$. In other words, if $\vec{c}$ is a superstable configuration and $\vec{d}$ is a valid configuration with $0 \leq d_i \leq c_i$ for all $i$, then $\vec{d}$ is also superstable. One can ask if a similar ``closure'' property holds for the superstable configurations of a signed graph $G_\phi$.  Unfortunately this property fails, as we illustrate with the following example.
    
\begin{ex}
We consider $G_\phi = -K_4$, the signed graph with underlying graph $K_4$, and with all edges not incident to the sink given a negative sign. Here the relevant matrices are

\[ L = \begin{pmatrix} 3 & 1 & 1 \\ 1 & 3 & 1 \\ 1 & 1 & 3 \end{pmatrix}, \quad \quad M = \begin{pmatrix} 3 & -1 & -1 \\ -1 & 3 & -1 \\ -1 & -1 & 3 \end{pmatrix}, \quad \quad LM^{-1} = \begin{pmatrix} 2 & \frac{3}{2} & \frac{3}{2} \\ \frac{3}{2} & 2 & \frac{3}{2} \\ \frac{3}{2} & \frac{3}{2}  & 2 \end{pmatrix}.\]

First consider the vector $\vec{c} = \ccvector{5, 6, 5}$. Note that $ML^{-1}\vec{c} = \ccvector{\frac{2}{5}, \frac{12}{5}, \frac{2}{5}}$, which implies that $\vec{c} \in S^+$. Also we see that $\lfloor ML^{-1}\vec{c} \rfloor = \ccvector{0,2,0}$ which is a superstable configuration for $K_4$. We conclude that $\vec{c}$ is $z$-superstable.

Next consider the configuration $\vec{d} = \ccvector{5, 5, 5}$. We have that $ML^{-1}\vec{d} = \ccvector{1, 1, 1}$, so that $\vec{d} \in S^+$.
%which can be seen to be stable since  $\vec{c} - L\vec{e}_i \not\in S^+$ for all $i = 1,2,3$.
However $\vec{d} - L\ccvector{1, 1, 1} = \ccvector{0, 0, 0}$, so that $d$ is not z-superstable as one can set fire all three nonsink vertices.

We have seen that $\vec{d}$ and $\vec{c}$ are both vectors in $S^+$ with the property that $0 \leq d_i \leq c_i$. At the same time, $\vec{c}$ is $z$-superstable whereas $\vec{d}$ is not.
\end{ex}

\subsection{Criticality and Firing the sink}\label{sec:SinkFire}
We next turn to the proof of Proposition \ref{prop:EOfSinkFiring0}, which provides yet another way to check for criticality of a configuration. 
In what follows, recall for a signed graph $G_\phi$ (with underlying graph $G$ and sink vertex $q$) we define the vector $\vec{s}$ as 
$$(\vec{s})_i := \begin{cases}1, & \text{ if }v_i \in G\text{ is a nonsink vertex adjacent to }q \\ 0,  & \text{ otherwise.}
\end{cases}$$ 

\begin{lem}\label{lem:sink}
For any signed graph $G_\phi$, we have that $LM^{-1}\vec{s} \in S^+$, so that $\vec{s} \in R^+$.
\end{lem}
\begin{proof}
By definition $\vec{s}$ has nonnegative entries, so we just need to check that $LM^{-1}\vec{s}$ has integer entries. But this follows from
\[LM^{-1} \vec{s} = LM^{-1} \left(\sum_{i: \; \{v_i, q\} \in E} M \vec{e}_i \right) = L \left(\sum_{i: \; \{v_i, q\} \in E} \vec{e}_i\right),\]
\noindent
which has integer entries since $L$ is integral.
\end{proof}

Note that the entries of $LM^{-1}\vec{s}$ also have a simple combinatorial interpretation, namely the $i$th coordinate is given by
\[(LM^{-1}\vec{s})_i =\begin{cases} 2 \;|\{\mbox{negative edges incident to $v_i$}\}|,& \mbox{if $\{v_i, q\} \not\in E$} \\ 2 \;|\{\mbox{negative edges incident to $v_i$}\}|+1,& \mbox{if $\{v_i, q\} \in E$.}\end{cases}\]

With this we can prove the following.

\begin{prop}\label{prop:EOfSinkFiring0}
   Suppose $G_\phi$ is a signed graph with chip-firing pair $(L,M)$. Then a configuration 
$\vec{c} \in S^+$ is critical if and only if $\mbox{stab}_{S^+}(\vec{c}+LM^{-1}\vec{s}) = \vec{c}$.
\end{prop}
\begin{proof}
We simply use a chain of if-and-only-ifs.
We first note that by Lemma \ref{lem:sink} we have that $\vec{c} + LM^{-1}\vec{s}$ is in $S^+$, and hence a valid configuration.

By Corollary \ref{cor:KeyIdea2}, we know that $\vec{c} \in S^+$ is critical for $G_{\phi}$ if and only if $\lfloor ML^{-1}\vec{c}\rfloor$ is a critical configuration for the underlying unsigned graph $G$. By \cite[Theorem 2.6.3]{Klivans}, $\lfloor ML^{-1}\vec{c}\rfloor$ is critical for $G$ if and only if the stabilization of $\lfloor ML^{-1}\vec{c}\rfloor+\vec{s}$ equals $\lfloor ML^{-1}\vec{c}\rfloor.$ Then, note that stabilization of $\lfloor ML^{-1}\vec{c}\rfloor+\vec{s}$ equals $\lfloor ML^{-1}\vec{c}\rfloor$ if and only if $\mbox{stab}_{R^+}( ML^{-1}\vec{c}+\vec{s})= ML^{-1}\vec{c}$ . Finally note that by Lemma \ref{lem:FiringSeq}, $\mbox{stab}_{R^+}( ML^{-1}\vec{c}+\vec{s})= ML^{-1}\vec{c}$ if and only if $\mbox{stab}_{S^+}( \vec{c}+LM^{-1}\vec{s})=\vec{c}$.
\end{proof}

We can also think about firing the sink $q$ in $G_\phi$, i.e. we consider the configuration $\vec{s}$ instead of $LM^{-1}\vec{s}$. We first point out that $\vec{s}$ is not necessarily valid for a general signed graph $G_\phi$. For instance, in the Example from Section \ref{sec:example} one can check that $\vec{s} = \cvector{1}{0}{1}$ and $ML^{-1}\vec{s} = \cvector{1}{-1}{1}$, so that $\vec{s} \notin S^+$. 

However, for certain graphs we can immediately see that $LM^{-1}\vec{s}$ is valid. In what follows for a signed graph $G_\phi$, we say that the sink $q$ is \emph{universal} if it is adjacent to all other vertices in $G$. A signed graph $G_\phi$ is \emph{signed-regular} if $(m_+, m_-)$ is the same for every vertex, where $m_+$ (resp. $m_-$) is the number of positive (resp. negative) edges incident to $v$.  We then have the following observation.

\begin{cor}\label{cor:SpecialSinkFire}
Let $G_{\phi}$ be a signed graph with universal sink $q$. Suppose $G_\phi$ is signed-regular with $m_-$ negative edges incident to each nonsink vertex. Then a valid configuration $\vec{c}$ is critical if and only if $\vec{c} = \stab(\vec{c} + (1 + 2m_-)\vec{1})$.
\end{cor} 

\begin{proof}
Since $q$ is universal we have that $\vec{s} = \vec{1}$, the vector of all $1$'s. Since $G_\phi$ is signed-regular, we have that $G \backslash q$ is regular and hence $M\vec{1} = \vec{1}$, so that $M^{-1}\vec{1} = \vec{1}$. Also $L \vec{1} = (2m_- +1)\vec{1}$, and hence $LM^{-1}\vec{s} = LM^{-1}\vec{1} = (2m_- +1)\vec{1}$. The result follows from Proposition~\ref{prop:EOfSinkFiring0}.
\end{proof}

Note in Corollary \ref{cor:SpecialSinkFire} we see that adding a multiple of $\vec{s}$ to a configuration $\vec{c}$ provides a way to check for criticality. For signed graphs that are not necessarily regular, we have a similar criteria, assuming $\vec{s}$ is valid.

\begin{prop}\label{prop:EOfSinkFiring}
Let $G_{\phi}$ be a signed graph with sink $q$. Suppose $\vec{s}$ is a valid configuration (so that $\vec{s} \in S^+$) and pick a positive integer $N_0$ such that $L^{-1}(N_{0} \vec{s}) \in \mathbb{Z}^n$. Then a configuration $\vec{c}\in S^+$ is critical if and only if $\mbox{stab}_{S^+}(\vec{c} + N_0\vec{s}) = \vec{c}$.
\end{prop}
\begin{proof}
Assume that $\vec{s} \in S^+$ and $L^{-1}(N_{0} \vec{s}) \in \mathbb{Z}^n$, and let $\vec{\ell} = ML^{-1}(N_0\vec{s})$. Note that since $\vec{s} \in S^+$ we have $ML^{-1} \vec{s} \in \R^n_{\geq 0}$ and hence $ML^{-1} \vec{s} \in \Z^n_{\geq 0}$.
Since $L^{-1}(N_{0} \vec{s}) \in \mathbb{Z}^n_{\geq 0}$, and $M = L_G$ is a usual graph Laplacian, we see that $\vec{\ell}$ is a so-called \emph{burning configuration} for $G$. 

Therefore we have by \cite[Proposition 2.6.27]{Klivans} that $\lfloor ML^{-1}\vec{c} \rfloor$ is a critical configuration of $G$ if and only if the $\stab_G(\lfloor ML^{-1}\vec{c} \rfloor+ \vec{\ell}) = \lfloor ML^{-1}\vec{c} \rfloor$. Thanks to Proposition~\ref{prop:keyFund} and using the same firing script, this is equivalent to $\mbox{stab}_{S^+}(\vec{c} + N_0\vec{s}) = \vec{c}$.
\end{proof}

Next we turn to the proof of Corollary \ref{cor:MaxCriticalConfig}, a generalization of the maximum stable configuration of unsigned graphs.

\begin{cor}\label{cor:MaxCriticalConfig}
Suppose $G_\phi$ is signed-regular with universal sink $q$, such that each nonsink vertex is incident to $m$ edges, with $m_-$ of them negative.  If $m^\prime = (m(2m_- + 1) - 1)$ then $m^\prime \vec{1}$ is a critical configuration.
Furthermore if $G_\phi = -G$, then any stable configuration $\vec{c}$ satisfies $\vec{c} \leq m^\prime\vec{1}$.
\end{cor}

\begin{proof}
First we assume that $G_\phi$ is signed-regular with $m$ edges incident to each nonsink vertex, $m_-$ of which are negative. As above we let $m' = m(2m_ - + 1)-1$. We first prove that $m'\vec{1}$ is a valid configuration, i.e. that $ML^{-1}(m'\vec{1}) \geq 0$. By the properties of $G,$ we obtain $$ML^{-1}(m'\vec{1}) = M\left(\frac{m'}{2m_-+1} \vec{1}\right)=\frac{m'}{2m_-+1} \vec{1} \geq \vec{0}.$$
 
%  This proves the validity of the configuration $m'\vec{1}$. %  Furthermore, the latter shows stability of $ML^{-1}(m'\vec{1})$ in $R^+$ and hence the stability of $m'\vec{1}$ in $S^+$.

To see that $m'\vec{1}$ is critical we observe $\lfloor ML^{-1}(m'\vec{1}) \rfloor = \lfloor \frac{m(2m_{-}+1)-1}{2m_-+1} \vec{1} \rfloor = (m-1)\vec{1}$. We then deduce that $\lfloor ML^{-1}(m\vec{1}) \rfloor =(m-1)\vec{1}$ is a critical configuration for $G$, noting that $\vec{0}$ is superstable and for $G$ we have $\vec{c}_{\max} = (m-1)\vec{1}$. By Theorem \ref{thm:KeyIdea} we conclude that $m'\vec{1}$ is critical for $G_\phi$

Now assume that $G_\phi = -G$, so that all edges not incident to the sink are negative (and hence $m_- = m-1$). We claim that any other stable configuration $\vec{c} \in S^+$  is componentwise smaller than $m^\prime \vec{1}$. For a contradiction, suppose there exists $\vec{c} \in S^+$ has an entry greater than $m' = m(2m_-+1) - 1$.

We claim there exists an entry of $L^{-1}\vec{c}$ that is at least $m$; we know this since otherwise all entries of $\vec{c} = L(L^{-1}\vec{c})$ would be less than $mm_-  + 
 m^2 = m(m_- + m) = m(2m_- + 1) = m'+1$. 
%Indeed, letting $i$ be the index of the largest entry of $L^{-1}\vec{c}$ and viewing $L^{-1}\vec{c}$ as a solution to the linear system $L\vec{x} = \vec{c}$, we find that $(L^{-1}\vec{c})_i \geq \frac{m(2m_-+1)}{m_-+1}\geq m$.
Now let $k = (L^{-1}\vec{c})_i$ be the largest entry of $(L^{-1}\vec{c})$ (so that in particular $k \geq m$), and let $\vec{v}$ denote the $i$th row of $M$. The dot product $\vec{v} \cdot (L^{-1}\vec{c})$ is then given by 
\[\vec{v} \cdot (L^{-1}\vec{c}) \geq mk - (m-1)k = k \geq m,\]
\noindent
since we are assuming that all entries of $L^{-1}\vec{c}$ are less than $k$. Hence the corresponding entry of $ML^{-1}\vec{c}$ is greater than $m$.  Since $\vec{c}$ is also valid, this implies that $ML^{-1}\vec{c} \in R^+$ is unstable as it can be fired at the $i$th vertex. This means that $\vec{c} \in S^+$ is unstable, and a contradiction is reached.
\end{proof}

\begin{ex}
    As an example, consider $-K_4$, the complete graph on 4 vertices where all edges not incident to the sink $q=4$ are negative. We see that $-K_4$ has a universal sink, and is signed regular with $m' = 14$. Hence Corollary \ref{cor:MaxCriticalConfig} implies that $\cvector{14}{14}{14}$ is a critical configuration that dominates all other stable configurations. In particular, the set of critical configurations is given by

    $$\biggl\{ \ccvector{14,14,14}, \ccvector{13,13,13},  \ccvector{12,12,12}, \ccvector{11,11,11}, \ccvector{11,11,10}, \ccvector{11,10, 11}, \ccvector{10,11,11}, \ccvector{10,10,10}, $$
    
 $$   \ccvector{10,10,9}, \ccvector{10,9,10},  \ccvector{9,10,10}, \ccvector{9,9,8}, \ccvector{9,8,9},  \ccvector{8, 9, 9}, \ccvector{8,8,7}, \ccvector{8,7,8}, \ccvector{7, 8, 8}, \ccvector{7,7,6}
    \ccvector{7,6,7}, \ccvector{6, 7, 7} \biggr\}.$$ 

\smallskip
\noindent We refer to Section \ref{sec:completeg} for more details regarding the critical groups of signed complete graphs.
\end{ex}

\begin{comment}
\color{blue} (Note: statement kind of sounds trivial in my opnion, but is it worth keeping?)
\begin{cor}\label{cor:zsuperstabprop}
Let $G_{\phi}$ be a signed graph with sink $q$, and let $G$ denote its underlying graph. Suppose $\vec{v} \in {\mathbb R}^n$ is a vector such that $ \vec{v} \leq \vec{w}$ for some superstable configuration $\vec{w}$ of $G$.  Then if $\vec{c}=LM^{-1}\vec{v}$ is valid we have that $\vec{c}$ is $z-$superstable.
\end{cor}

\begin{proof}
Suppose $\vec{v} \in {\mathbb R}^n$ is a valid configuration in $R^+$ such that $\vec{v}$ is coordinate-wise less than some superstable configuration $\vec{\gamma}$ of $G$. Since, $\vec{\gamma}$ is superstable, $\vec{\gamma}-M\vec{z} \not\geq \vec{0}$ for any nontrivial $\vec{z} \in \mathbb{Z}^n$. Because of this and how $\lfloor \vec{v} \rfloor \leq \vec{v} \leq \vec{\gamma}$, we find that $\lfloor \vec{v} \rfloor$ is also a superstable configuration of $G$. Corollary \ref{thm:KeyIdea} thus tells us that $LM^{-1}\vec{v}$ is a $z-$superestable configuration of chips in $G_{\phi}$. 
\end{proof}
\end{comment}

\section{Critical groups of signed graphs}\label{sec:groups}

We next study critical groups of signed graphs. Our first observation is a connection between the identity of the critical group of a signed graph $G_{\phi}$, and that of its underlying graph $G$.

\begin{thm}\label{thm:CriticalityIdentity}
Suppose $G_{\phi}$ is a signed graph with underlying graph $G$, and let $L = L_{G_\phi}$ and $M = L_G$ denote the respective Laplacians.  Let $\vec{\epsilon}_G$ denote the identity of the critical group ${\mathcal K}(G)$. Then the identity of ${\mathcal K}(G_\phi)$ is given by
\[\vec{\epsilon}_{G_\phi} = LM^{-1} \vec{\epsilon}_G.\]
\end{thm}

\begin{proof}
 We must show that $\vec{\epsilon}_{G_\phi}$ is valid,  stable and reachable, and that $[\vec{e}_{G_{\phi}}]$ has order $1$ in ${\mathcal K}(G_\phi)$, so that $\stab_{R^+}(\vec{\epsilon}_{G_\phi}+ \vec{\epsilon}_{G_\phi}) = \vec{\epsilon}_{G_\phi}$. 

Since $\vec{\epsilon}_G\sim_M \vec{0},$ we have $\vec{\epsilon}_G=M\vec{v}$ for some $\vec{v}\in \mathbb{Z}^{n}$. Thus, $LM^{-1} \vec{\epsilon}_G= LM^{-1}\left(M\vec{v}\right) = L\vec{v} \in \mathbb{Z}^{n},$ since $L$ is an integer matrix. This gives us $LM^{-1} \vec{\epsilon}_G \in S^+$. Now from Lemma~\ref{lem:GuzKlivansSR} and Lemma~\ref{lem:FiringSeq}, it suffices to show that $\vec{\epsilon}_G$ is reachable, since it being stable and having order $1$ follows directly from the fact that it is the identity of $\mathcal{K}(G)$. To see that $\vec{\epsilon}_G$ is reachable, start from $\vec{\epsilon}_G$ and fire the sink suitably many times to get a valid configuration with the property that it can be fired at every non-sink vertex. 
\end{proof}

We next show that the critical group is invariant under vertex switching. After posting a version of our paper to the ArXiv we were made aware that this observation also appeared as Corollary 44 in \cite{Adinkras}.

\begin{lem}\label{lem:SmithBalanced}
Suppose $G_\phi$ and $G_\psi$ are signed graphs that are switching equivalent.  Then we have an isomorphism of critical groups
\[{\mathcal K}(G_\phi) \cong {\mathcal K}(G_\psi).\]
\end{lem}

\begin{proof}

It suffices to prove that claim for a single vertex switching. For this suppose $G_\psi$ is obtained by switching the vertex of $G_\phi$ corresponding to row $i$ and column $i$ of the Laplacian $L_{G_\phi}$. Hence the matrix $L_{G_\psi}$ is obtained from $L_{G_\phi}$ by changing each -1 in row $i$ and column $i$ to a 1, and vice versa. This is equivalent to multiplying both row $i$ and column $i$ by $-1$ (which leaves the $i$th diagonal entry unchanged). Since $-1$ is a unit in ${\mathbb Z}$ these operations preserve the Smith normal form. The result follows from Remark \ref{snf}.
\end{proof}

\begin{cor}\label{cor:treegroup}
Suppose $G_\phi$ is a connected signed graph with sink $q$ such that $G \backslash q$ is a tree. Then there exists an isomorphism of critical groups
\[{\mathcal K}(G_\phi) \cong {\mathcal K}(|G|).\]
\end{cor}

\begin{proof}
    This follows from Lemmas \ref{lem:SmithBalanced} and \ref{lem: switchingclass}.
\end{proof}

\subsection{Cycle Graphs}\label{sec:cycle}

%%%%%%%%%%%%%%%%%%%%%%%%%%%%%

Having developed some results regarding chip-firing on signed graphs, we next turn our attention to some concrete examples. We first consider cycles.

We let $C_{n}$ denote the cycle graph on $n$ vertices with sink vertex $q = v_n$ and nonsink vertices $v_1, \dots, v_{n-1}$. The reduced Laplacian of $C_{n}$ is given by the $(n-1) \times (n-1)$ matrix $M = L_{C_{n}}$, where $$M_{ij} = \begin{cases} 2, & \mbox{$i = j$} \\ -1, & \mbox{$|i-j| = 1$}  \\ 0, & \mbox{otherwise.}\end{cases}$$ 
% The reduced Laplacian of any signed cycle $(C_{n})_\phi$ will be obtained by changing the off-diagonal signs corresponding to negative edges not adjacent to the sink.

\begin{figure}[h]
 \begin{center}
 
 % \vspace{3.7mm}
   \scalebox{1.0}{    
\tikzset{every picture/.style={line width=0.75pt}} %set default line width to 0.75pt        

\begin{tikzpicture}[x=0.65pt,y=0.65pt,yscale=-1,xscale=1]

%Straight Lines [id:da3146019219865912] 
\draw [color={rgb, 255:red, 208; green, 2; blue, 27 }  ,draw opacity=1 ][line width=0.75]    (401.72,143.44) -- (384.33,172.9) -- (364.58,207.29) ;
%Straight Lines [id:da7973786605456414] 
\draw [color={rgb, 255:red, 208; green, 2; blue, 27 }  ,draw opacity=1 ][line width=0.75]    (364.58,77.05) -- (330.19,77.05) -- (290.28,77.14) ;
%Shape: Ellipse [id:dp005058835243594517] 
\draw  [fill={rgb, 255:red, 0; green, 0; blue, 0 }  ,fill opacity=1 ] (357.94,207.29) .. controls (357.94,203.59) and (360.91,200.58) .. (364.58,200.58) .. controls (368.24,200.58) and (371.21,203.59) .. (371.21,207.29) .. controls (371.21,211) and (368.24,214) .. (364.58,214) .. controls (360.91,214) and (357.94,211) .. (357.94,207.29) -- cycle ;
%Shape: Ellipse [id:dp23500792922731817] 
\draw  [fill={rgb, 255:red, 0; green, 0; blue, 0 }  ,fill opacity=1 ] (357.94,77.14) .. controls (357.94,73.44) and (360.91,70.43) .. (364.58,70.43) .. controls (368.24,70.43) and (371.21,73.44) .. (371.21,77.14) .. controls (371.21,80.85) and (368.24,83.85) .. (364.58,83.85) .. controls (360.91,83.85) and (357.94,80.85) .. (357.94,77.14) -- cycle ;
%Shape: Ellipse [id:dp9922244075633255] 
\draw  [fill={rgb, 255:red, 0; green, 0; blue, 0 }  ,fill opacity=1 ] (395.09,142.22) .. controls (395.09,138.51) and (398.06,135.51) .. (401.73,135.51) .. controls (405.39,135.51) and (408.36,138.51) .. (408.36,142.22) .. controls (408.36,145.92) and (405.39,148.93) .. (401.73,148.93) .. controls (398.06,148.93) and (395.09,145.92) .. (395.09,142.22) -- cycle ;
%Shape: Ellipse [id:dp8357231728050818] 
\draw  [fill={rgb, 255:red, 0; green, 0; blue, 0 }  ,fill opacity=1 ] (283.65,207.29) .. controls (283.65,203.59) and (286.62,200.58) .. (290.28,200.58) .. controls (293.95,200.58) and (296.92,203.59) .. (296.92,207.29) .. controls (296.92,211) and (293.95,214) .. (290.28,214) .. controls (286.62,214) and (283.65,211) .. (283.65,207.29) -- cycle ;
%Straight Lines [id:da40296593907132405] 
\draw    (364.58,77.14) -- (401.73,142.22) ;
%Straight Lines [id:da06161185964561411] 
\draw    (290.28,207.29) -- (294,207.29) -- (364.58,207.29) ;
%Straight Lines [id:da7028776235000587] 
\draw    (253.13,142.22) -- (290.28,207.29) ;
%Straight Lines [id:da673016078945841] 
\draw [color={rgb, 255:red, 208; green, 2; blue, 27 }  ,draw opacity=1 ]   (253.13,142.22) -- (290.28,77.14) ;
%Shape: Ellipse [id:dp34485817031132826] 
\draw  [fill={rgb, 255:red, 0; green, 0; blue, 0 }  ,fill opacity=1 ] (283.65,77.14) .. controls (283.65,73.44) and (286.62,70.43) .. (290.28,70.43) .. controls (293.95,70.43) and (296.92,73.44) .. (296.92,77.14) .. controls (296.92,80.85) and (293.95,83.85) .. (290.28,83.85) .. controls (286.62,83.85) and (283.65,80.85) .. (283.65,77.14) -- cycle ;
%Shape: Ellipse [id:dp3461367805519411] 
\draw  [fill={rgb, 255:red, 0; green, 0; blue, 0 }  ,fill opacity=1 ] (246.5,142.22) .. controls (246.5,138.51) and (249.47,135.51) .. (253.13,135.51) .. controls (256.8,135.51) and (259.77,138.51) .. (259.77,142.22) .. controls (259.77,145.92) and (256.8,148.93) .. (253.13,148.93) .. controls (249.47,148.93) and (246.5,145.92) .. (246.5,142.22) -- cycle ;

% Text Node
\draw (266.8,202.62) node [anchor=north west][inner sep=0.75pt]    {$q$};
% Text Node
\draw (386.33,172.4) node [anchor=north west][inner sep=0.75pt]    {$-$};
% Text Node
\draw (253,94.4) node [anchor=north west][inner sep=0.75pt]    {$-$};
% Text Node
\draw (388,94.4) node [anchor=north west][inner sep=0.75pt]    {$+$};
% Text Node
\draw (250,170.4) node [anchor=north west][inner sep=0.75pt]    {$+$};
% Text Node
\draw (319,208.4) node [anchor=north west][inner sep=0.75pt]    {$+$};
% Text Node
\draw (320,57.4) node [anchor=north west][inner sep=0.75pt]    {$-$};

\end{tikzpicture}}
    \captionsetup{width=1.0\linewidth}
  \captionof{figure}{A signed cycle $(C_6)_\phi$.}
  \label{fig:signedC_6}

 \end{center}
\end{figure}
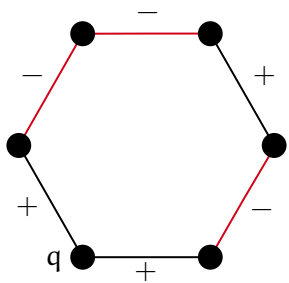

\begin{prop}\label{prop:Cyclecrit}
Suppose $(C_{n})_\phi$ is any signed cycle on $n$ vertices. Then the critical group satisfies
\[\mathcal{K}((C_{n})_\phi) \cong \mathbb{Z}_{n}.\]
\end{prop}

\begin{proof}
Note that $C_n \backslash q$ is a path.  Hence by Corollary \ref{cor:treegroup}, the critical group is the same as that of $C_n$, which is known to be ${\mathbb Z}_{n}$ (see \cite[Theorem 4.5.6]{Klivans}).
\end{proof}

Recall that for an (unsigned) graph $G$, there is a simple duality between the critical configurations and superstable configurations. In particular, a configuration $\vec{c}$ is critical if and only if $\vec{c}_{\text{max}} - \vec{c}$ is superstable, where $\vec{c}_{\text{max}}$ is the maximum stable configuration defined in Section~\ref{sec:introduction section}.

In general it seems difficult to describe a similar duality for the case of signed graphs. For the case of the all-negative odd cycle $-C_{2m+1}$, however, we can recover such a result as described in Theorem \ref{thm:PetitDuality}. To establish this duality, we employ a statistic on configurations of $-C_{2m+1}$ that is invariant under chip-firing and leads to an explicit description of the critical group ${\mathcal K}(-C_{2m+1})$.  The construction is inspired by a similar statistic for unsigned cycles discussed in \cite[Problem 1.2]{CorryPerk}. See Figure \ref{fig:statistic} for an illustration of the case of $-C_5$.

\begin{figure}[h]
 \begin{center}
 
%  \vspace{3.7mm}
    \input{tiksstatistic}
    \captionsetup{width=1.0\linewidth}
  \captionof{figure}{Example of our statistic on signed odd cycles.}
  \label{fig:statistic}

 \end{center}
\end{figure}

\begin{lem}\label{lem:SuhoStatistic}
The map $f: {\mathcal K}(-C_{2m+1}) \to \mathbb{Z}_{2m+1}$ defined by $$\begin{pmatrix}x_1 \\ \vdots \\ x_{2m}\end{pmatrix} + \img L_{-C_{2m+1}}  \mapsto \sum_{j=1}^{2m}(-1)^{j+1}jx_j \pmod{2m+1}$$ is an isomorphism.
\end{lem}
\begin{proof} 

To see that $f$ is well defined, any firing of a vertex (which decreases a coordinate by $2$ and decreases the adjacent coordinate(s) by $1$ each) preserves the value of the expression $\sum_{j=1}^{2m}(-1)^{j+1}jx_j \pmod{2m+1}$. It is also clear that $f$ is a homomorphism since the expression on the right-hand side is a linear expression. The map $f$ is also surjective since for any $a \in \mathbb{Z}_{2m+1}$, we have $a(\vec{e}_1)$ in its pre-image. Combined with $|{\mathcal K}(-C_{2m+1})| = 2m+1$ from Proposition \ref{prop:Cyclecrit}, we conclude that $f$ is an isomorphism.
\end{proof}

We next show that all $z$-superstable configurations of $-C_{2m+1}$ share a common property. We say that $\vec{c} \in {\mathbb Z}^{2m}$ is \emph{palindromic} if for all $i = 1, \dots, 2m$ we have 
$(\vec{c})_{2m+1-i} = (\vec{c})_{i}$.

%\begin{defn} A vector $\vec{c} \in {\mathbb Z}^{2m}$ is said to be %\emph{palindromic} if for all $i = 1, \dots, 2n$ we have 
%$$(\vec{c})_{2m+1-i} = (\vec{c})_{i}.$$
% \end{defn}

\begin{lem}\label{lem:symmetry}
For all $m \geq 1$, each $z$-superstable configuration of $-C_{2m+1}$ is palindromic.
\end{lem}
\begin{proof}
Assume $\vec{c}$ is $z$-superstable and let $\vec{d}$ be the configuration obtained from $\vec{c}$ by flipping the coordinates: $(\vec{d})_{2m+1-i} = (\vec{c})_i$ for all $i$. By symmetry of $-C_{2m+1}$, we see that $\vec{d}$ is also $z$-superstable.
Using $f$ from Lemma \ref{lem:SuhoStatistic} we have that $f(\vec{d}) = f(\vec{c})$ since
\[(-1)^i a \equiv (-1)^{2m+1-i} (2m+1-a) \pmod{2m+1}\]
\noindent
for any $i$ and $a \in \mathbb{Z}_{2m+1}$. Therefore $\vec{c}$ and $\vec{d}$ are in the same equivalence class defined by $\img L_{-C_{2m+1}}$, which by Theorem~\ref{thm:class} implies that they are equal. Hence $\vec{c}$ has to be palindromic.
\end{proof}

Finally we are ready to provide the proof of the duality.

\begin{thm}\label{thm:PetitDuality}
Suppose $m\geq 1$ is an integer. Then the set map $$\Delta: \{\mbox{z-superstable configurations of $-C_{2m+1}$}\} \to \{\mbox{critical configurations of $-C_{2m+1}$}\}$$ given by $\vec{c}     \mapsto \vec{\epsilon}_{-C_{2m+1}} + \vec{c}$ is a bijection.
\end{thm}

\begin{proof}
Let $\vec{c}$ be a $z$-superstable configuration. Our goal is to show that $\vec{d} := \vec{e}_{-C_{2m+1}} + \vec{c}$ is critical.

We first show that $\vec{d}$ is stable. Note that by Theorem~\ref{thm:CriticalityIdentity} we have that $ML^{-1}\vec{e}_{-C_{2m+1}} = \vec{e}_{C_{2m+1}}$. Using the method of \cite[Problem 1.2]{CorryPerk} one can check that $\vec{e}_{C_{2m+1}} = \vec{1}$. Recall that $\vec{c}$ is palindromic via Lemma \ref{lem:symmetry}, and since $-C_{2m+1}$ is symmetric we have that $ML^{-1}\vec{c}$ is also palindromic. 

From $z$-superstability of $ML^{-1}\vec{c}$, it follows that set-firing $\{v_j, \dots, v_{(2m+1)-j}\}$ will not result in an allowable configuration for any $1 \leq j \leq m$. 

If $j \neq 1$ we have $$ML^{-1}\vec{c} - M\left(\sum_{s=j}^{(2m+1)-j}\vec{e}_s\right)=ML^{-1}\vec{c}-\left(\vec{e}_j+\vec{e}_{(2m+1)-j}\right)+\left(\vec{e}_{j-1}+\vec{e}_{(2m+1)-(j-1)}\right)\not\geq 0,$$ 

If $j = 1$ we have $$ML^{-1}\vec{c} - M\left(\sum_{s=j}^{(2m+1)-j}\vec{e}_s\right)=ML^{-1}\vec{c}-\left(\vec{e}_j+\vec{e}_{(2m+1)-j}\right)\not\geq 0.$$

This gives us $(ML^{-1}\vec{c})_i=(ML^{-1}\vec{c})_{(2m+1)-i}<1$ for all $i$. This implies $(ML^{-1}\vec{d})_i < 2$ for all $i$, giving us the stability of $\vec{d}$.

This implies $(ML^{-1}\vec{c})_i < 2$ for all $i$ and hence the stability of $\vec{d}$.

It only remains to show the reachability of $\vec{d}$. Since $\vec{e}_{-C_{2m+1}}$ is reachable, we can find some $\vec{b}$ that can be fired at all vertices and stabilizes to $\vec{e}_{-C_{2m+1}}$. Since $\vec{c} \in S^+$, we have that $\vec{b} + \vec{c}$ can be fired at all vertices and stabilizes to $\vec{e}_{-C_{2m+1}} + \vec{c}  = \vec{d}$ using the same firing sequence. 

such that $\stab(\vec{b}) = \vec{e}_{-C_{2m+1}}$ and such that $ML^{-1}\vec{b}-M\vec{e}_i\geq \vec{0}$ for all $i$. Now since $\vec{c} \in S^+$, we have that $\vec{b} + \vec{c}$ stabilizes to $\vec{d}$ and is firable everywhere

Recall that $ML^{-1}\vec{c} \geq \vec{0}$ follows from $\vec{c} \in S^+$. So we get $ML^{-1}\left(\vec{b}+\vec{c}\right)-M\vec{e}_i \geq \vec{0}$ for all $i$, meaning that $\vec{b} + \vec{c}$ is a configuration where all vertices can be fired. Moreover since we reach $\vec{e}_{-C_{2n+1}}$ by stabilizing $\vec{b}$ starting from, we can replicate this firing sequence to reach $\vec{d}$ as stabilization of $\vec{b} + \vec{c}$. This proves the claim.

\end{proof}

For even $n$, the map $f$ employed in Lemma \ref{lem:SuhoStatistic} exhibits less predictable behavior and we have not been able to demonstrate a similar duality. Hence we ask the following.

\begin{question}
Does there exist a easily described duality between the set of critical and $z$-superstable configurations of the negative cycles $-C_{n}$ for the case $n$ is even?
\end{question}

\subsection{Complete Graphs}\label{sec:completeg}

We next consider signed complete graphs $(K_n)_{\phi}$. In general, the critical group of such graphs seems to be quite complicated. Despite that, there are some cases where the critical group is fairly simple to describe.

Recall that from Lemma~\ref{lem:SmithBalanced}, the critical group only depends on the switching class of the signed graph. We start with the obvious case, of when $(K_n)_{\phi}$ is switching equivalent to the unsigned complete graph $K_n$. A version of Cayley's Theorem \cite[Theorem 4.5.7]{Klivans} tells us that ${\mathcal K}((K_{n})_{\phi}) \cong {\mathcal K}(K_{n}) \cong \mathbb{Z}_{n}^{n-2}$.
We next turn our attention to signed complete graphs that are switching equivalent to $-K_n$, the \emph{negative complete graph} where all edges not incident to sink are negative. See Figure \ref{fig:signedcomplete} for an example.

\begin{figure}[h]
 \begin{center}
 
%  \vspace{3.7mm}
    \tikzset{every picture/.style={line width=0.75pt}} %set default line width to 0.75pt        

\begin{tikzpicture}[x=0.75pt,y=0.75pt,yscale=-1,xscale=1]

%Straight Lines [id:da6650056955213772] 
\draw    (126.87,242.4) -- (227.84,169.41) ;
%Straight Lines [id:da5454341669246281] 
\draw    (126.87,242.4) -- (165.66,124) ;
%Shape: Ellipse [id:dp4665983868443635] 
\draw  [fill={rgb, 255:red, 0; green, 0; blue, 0 }  ,fill opacity=1 ] (120.24,242.4) .. controls (120.24,238.69) and (123.21,235.69) .. (126.87,235.69) .. controls (130.53,235.69) and (133.5,238.69) .. (133.5,242.4) .. controls (133.5,246.1) and (130.53,249.11) .. (126.87,249.11) .. controls (123.21,249.11) and (120.24,246.1) .. (120.24,242.4) -- cycle ;
%Straight Lines [id:da39019496494945516] 
\draw [color={rgb, 255:red, 208; green, 2; blue, 27 }  ,draw opacity=1 ]   (165.66,124) -- (227.84,169.41) ;
%Straight Lines [id:da8498156677735282] 
\draw [color={rgb, 255:red, 208; green, 2; blue, 27 }  ,draw opacity=1 ]   (103.26,169.11) -- (165.66,124) ;
%Straight Lines [id:da2821348034278177] 
\draw [color={rgb, 255:red, 0; green, 0; blue, 0 }  ,draw opacity=1 ]   (203.87,242.58) -- (227.84,169.41) ;
%Straight Lines [id:da28643484976671196] 
\draw    (126.87,242.4) -- (203.87,242.58) ;
%Straight Lines [id:da33925035943273096] 
\draw    (103.26,169.11) -- (126.87,242.4) ;
%Straight Lines [id:da6153908307593274] 
\draw [color={rgb, 255:red, 208; green, 2; blue, 27 }  ,draw opacity=1 ]   (103.26,169.11) -- (227.84,169.41) ;
%Straight Lines [id:da8708796376814909] 
\draw [color={rgb, 255:red, 0; green, 0; blue, 0 }  ,draw opacity=1 ]   (103.26,169.11) -- (203.87,242.58) ;
%Straight Lines [id:da16223647279456577] 
\draw [color={rgb, 255:red, 0; green, 0; blue, 0 }  ,draw opacity=1 ]   (203.87,242.58) -- (165.66,124) ;
%Shape: Ellipse [id:dp5044636465214056] 
\draw  [fill={rgb, 255:red, 0; green, 0; blue, 0 }  ,fill opacity=1 ] (221.21,169.41) .. controls (221.21,165.71) and (224.18,162.7) .. (227.84,162.7) .. controls (231.51,162.7) and (234.48,165.71) .. (234.48,169.41) .. controls (234.48,173.12) and (231.51,176.12) .. (227.84,176.12) .. controls (224.18,176.12) and (221.21,173.12) .. (221.21,169.41) -- cycle ;
%Shape: Ellipse [id:dp2147900730094945] 
\draw  [fill={rgb, 255:red, 0; green, 0; blue, 0 }  ,fill opacity=1 ] (197.24,242.58) .. controls (197.24,238.88) and (200.21,235.88) .. (203.87,235.88) .. controls (207.53,235.88) and (210.5,238.88) .. (210.5,242.58) .. controls (210.5,246.29) and (207.53,249.29) .. (203.87,249.29) .. controls (200.21,249.29) and (197.24,246.29) .. (197.24,242.58) -- cycle ;
%Shape: Ellipse [id:dp987926550731087] 
\draw  [fill={rgb, 255:red, 0; green, 0; blue, 0 }  ,fill opacity=1 ] (159.03,124) .. controls (159.03,120.29) and (162,117.29) .. (165.66,117.29) .. controls (169.32,117.29) and (172.29,120.29) .. (172.29,124) .. controls (172.29,127.71) and (169.32,130.71) .. (165.66,130.71) .. controls (162,130.71) and (159.03,127.71) .. (159.03,124) -- cycle ;
%Shape: Ellipse [id:dp6431508676265711] 
\draw  [fill={rgb, 255:red, 0; green, 0; blue, 0 }  ,fill opacity=1 ] (96.62,169.11) .. controls (96.62,165.4) and (99.59,162.4) .. (103.26,162.4) .. controls (106.92,162.4) and (109.89,165.4) .. (109.89,169.11) .. controls (109.89,172.81) and (106.92,175.82) .. (103.26,175.82) .. controls (99.59,175.82) and (96.62,172.81) .. (96.62,169.11) -- cycle ;

% Text Node
\draw (119.65,252.54) node [anchor=north west][inner sep=0.75pt]    {$q$};
% Text Node
\draw (124,200.4) node [anchor=north west][inner sep=0.75pt]    {$+$};
% Text Node
\draw (143,208.4) node [anchor=north west][inner sep=0.75pt]    {$+$};
% Text Node
\draw (102,202.4) node [anchor=north west][inner sep=0.75pt]    {$+$};
% Text Node
\draw (157,242.4) node [anchor=north west][inner sep=0.75pt]    {$+$};
% Text Node
\draw (173,207.4) node [anchor=north west][inner sep=0.75pt]    {$+$};
% Text Node
\draw (194,200.4) node [anchor=north west][inner sep=0.75pt]    {$+$};
% Text Node
\draw (216,202.4) node [anchor=north west][inner sep=0.75pt]    {$+$};
% Text Node
\draw (159,152.4) node [anchor=north west][inner sep=0.75pt]    {$-$};
% Text Node
\draw (122,132.4) node [anchor=north west][inner sep=0.75pt]    {$-$};
% Text Node
\draw (198,132.4) node [anchor=north west][inner sep=0.75pt]    {$-$};

\end{tikzpicture}
    \captionsetup{width=1.0\linewidth}
  \captionof{figure}{A signed complete graph $(K_5)_\phi$ that is switching equivalent to $-K_5$.}\label{fig:signedcomplete}

 \end{center}
 \end{figure}
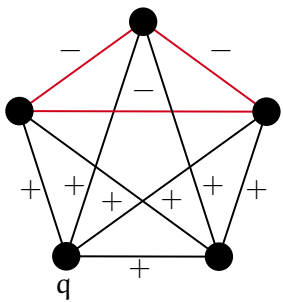

\begin{prop}\label{prop:Completecrit}
Suppose $n \geq 2$ and let $(K_{n})_\phi$ be a signed complete graph with sink $q$ that is switching equivalent to $-K_{n}$. 
Then 
\[{\mathcal K}((K_{n})_\phi) \cong \mathbb{Z}_{n-2}^{n-3} \oplus  \mathbb{Z}_{(n-2)(2n-3)}.\]
\end{prop}

\begin{proof}
As noted above, it suffices to compute the critical group of $-K_{n}$. For this note that the signed Laplacian $L = L_{-K_{n}}$ is given by $(n-2)I_{n-1} - (-1)J_{n-1}$ where $I_{n-1}$ and $J_{n-1}$ are the $(n-1) \times (n-1)$ identity matrix and the all-ones matrix, respectively.
% The critical group of the graph with - signs assigned to all edges not adjacent to the sink is $\mathcal{K}(-G) = \mathbb{Z}^{n-1} /\mbox{im }L_{-G}$ with $L_{-G}$ being the reduced Laplacian of $-G$. 
 We apply the argument from \cite[Corollary 7.6]{Reiner}, which tells us that the invariant factors of $L$ are $1, n-2,  (n-2)(2n-3)$ with $n-2$ having multiplicity of $n-3$.  From this it follows that $\mathbb{Z}^{n}/\mbox{im }L \cong \mathbb{Z}_{n-2}^{n-3} \oplus  \mathbb{Z}_{(n-2)(2n-3)}$, which completes the proof.
\end{proof}

We remark that $(n-2)(2n-3)$ are known as the \emph{second hexagonal numbers} \cite{A014105}.

\subsection{Wheel Graphs}\label{sec:wheel}
In this section we prove Theorem \ref{thm:wheelresult}, which describes the critical groups of signed wheel graphs. We begin by introducing the wheel graph and some terminology.

\begin{defn} For $n \geq 3$, the \emph{wheel graph} $W_n$ is the graph with $n+1$ vertices formed by connecting a universal sink vertex $q$ to all vertices of an $n$-cycle. 
\end{defn}

We will use $v_1,\ldots,v_n$ to label the nonsink vertices of $W_n$. See Figure \ref{fig:wheelgraph} for an example. The reduced Laplacian of the (unsigned) wheel $W_n$
 is the $n \times n$ matrix given by $$(L_{W_n})_{ij} = \begin{cases} 3, & \mbox{$i = j$} \\ -1, & \mbox{$i \equiv (j+1)$ or $i \equiv (j-1) \Mod{n}$} \\ 0, & \mbox{otherwise}. \end{cases}$$
 
 The reduced Laplacian of a signed wheel $(W_n)_\phi$ is obtained from changing the sign of the nondiagonal entries corresponding to negative edges. 
As usual let $-W_{n}$ denote the signed wheel graph where all edges not incident to the sink are negative.

\begin{figure}[h]
 \begin{center}
 
%  \vspace{3.7mm}
    \tikzset{every picture/.style={line width=0.75pt}} %set default line width to 0.75pt        

\begin{tikzpicture}[x=0.75pt,y=0.75pt,yscale=-1,xscale=1]

%Shape: Regular Polygon [id:dp5965108520764912] 
\draw   (363,138.5) -- (330.25,195.22) -- (264.75,195.22) -- (232,138.5) -- (264.75,81.78) -- (330.25,81.78) -- cycle ;
%Straight Lines [id:da6928817439703425] 
\draw    (330.25,195.22) -- (264.75,81.78) ;
%Straight Lines [id:da1883256250171117] 
\draw    (264.75,195.22) -- (330.25,81.78) ;
%Straight Lines [id:da05630781979907873] 
\draw    (232,138.5) -- (363,138.5) ;
%Shape: Circle [id:dp867760595013449] 
\draw  [fill={rgb, 255:red, 0; green, 0; blue, 0 }  ,fill opacity=1 ] (228.2,138.5) .. controls (228.2,136.4) and (229.9,134.7) .. (232,134.7) .. controls (234.1,134.7) and (235.8,136.4) .. (235.8,138.5) .. controls (235.8,140.6) and (234.1,142.3) .. (232,142.3) .. controls (229.9,142.3) and (228.2,140.6) .. (228.2,138.5) -- cycle ;
%Shape: Circle [id:dp34800411497245] 
\draw  [fill={rgb, 255:red, 0; green, 0; blue, 0 }  ,fill opacity=1 ] (260.95,195.22) .. controls (260.95,193.12) and (262.65,191.42) .. (264.75,191.42) .. controls (266.85,191.42) and (268.55,193.12) .. (268.55,195.22) .. controls (268.55,197.33) and (266.85,199.03) .. (264.75,199.03) .. controls (262.65,199.03) and (260.95,197.33) .. (260.95,195.22) -- cycle ;
%Shape: Circle [id:dp18329218662298574] 
\draw  [fill={rgb, 255:red, 0; green, 0; blue, 0 }  ,fill opacity=1 ] (326.45,195.22) .. controls (326.45,193.12) and (328.15,191.42) .. (330.25,191.42) .. controls (332.35,191.42) and (334.05,193.12) .. (334.05,195.22) .. controls (334.05,197.33) and (332.35,199.03) .. (330.25,199.03) .. controls (328.15,199.03) and (326.45,197.33) .. (326.45,195.22) -- cycle ;
%Shape: Circle [id:dp19026003444610606] 
\draw  [fill={rgb, 255:red, 0; green, 0; blue, 0 }  ,fill opacity=1 ] (260.95,81.78) .. controls (260.95,79.67) and (262.65,77.97) .. (264.75,77.97) .. controls (266.85,77.97) and (268.55,79.67) .. (268.55,81.78) .. controls (268.55,83.88) and (266.85,85.58) .. (264.75,85.58) .. controls (262.65,85.58) and (260.95,83.88) .. (260.95,81.78) -- cycle ;
%Shape: Circle [id:dp866440108811211] 
\draw  [fill={rgb, 255:red, 0; green, 0; blue, 0 }  ,fill opacity=1 ] (326.45,81.78) .. controls (326.45,79.67) and (328.15,77.97) .. (330.25,77.97) .. controls (332.35,77.97) and (334.05,79.67) .. (334.05,81.78) .. controls (334.05,83.88) and (332.35,85.58) .. (330.25,85.58) .. controls (328.15,85.58) and (326.45,83.88) .. (326.45,81.78) -- cycle ;
%Shape: Circle [id:dp30646924854797275] 
\draw  [fill={rgb, 255:red, 0; green, 0; blue, 0 }  ,fill opacity=1 ] (359.2,138.5) .. controls (359.2,136.4) and (360.9,134.7) .. (363,134.7) .. controls (365.1,134.7) and (366.8,136.4) .. (366.8,138.5) .. controls (366.8,140.6) and (365.1,142.3) .. (363,142.3) .. controls (360.9,142.3) and (359.2,140.6) .. (359.2,138.5) -- cycle ;
%Shape: Circle [id:dp06833941681823097] 
\draw  [fill={rgb, 255:red, 0; green, 0; blue, 0 }  ,fill opacity=1 ] (293.7,138.5) .. controls (293.7,136.4) and (295.4,134.7) .. (297.5,134.7) .. controls (299.6,134.7) and (301.3,136.4) .. (301.3,138.5) .. controls (301.3,140.6) and (299.6,142.3) .. (297.5,142.3) .. controls (295.4,142.3) and (293.7,140.6) .. (293.7,138.5) -- cycle ;

% Text Node
\draw (292,149.4) node [anchor=north west][inner sep=0.75pt]    {$q$};
% Text Node
\draw (211,131.4) node [anchor=north west][inner sep=0.75pt]    {$v_{6}$};
% Text Node
\draw (259,63.4) node [anchor=north west][inner sep=0.75pt]    {$v_{1}$};
% Text Node
\draw (324,63.4) node [anchor=north west][inner sep=0.75pt]    {$v_{2}$};
% Text Node
\draw (368,132.4) node [anchor=north west][inner sep=0.75pt]    {$v_{3}$};
% Text Node
\draw (332.25,202.82) node [anchor=north west][inner sep=0.75pt]    {$v_{4}$};
% Text Node
\draw (251.25,202.82) node [anchor=north west][inner sep=0.75pt]    {$v_{5}$};

\end{tikzpicture}
    \captionsetup{width=1.0\linewidth}
  \captionof{figure}{The wheel graph $W_6$.}
  \label{fig:wheelgraph}

 \end{center}
\end{figure}
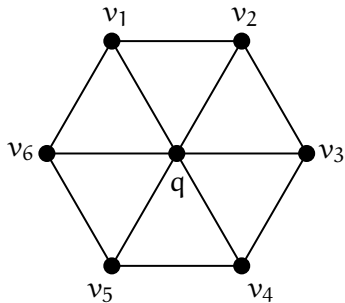

 In \cite{Biggs} Biggs computed ${\mathcal K}(W_n)$ for all $n$, and to recall his results we review some well-known integer sequences.
\begin{defn}\label{defn:Fib}
We let $f_n$ denote the $n$th \emph{Fibonacci number}, defined recursively by $f_{n} = f_{n-1} + f_{n-2}$ with base cases $f_0=0, f_1=1$.
\end{defn}
\begin{defn}\label{defn:Luc}
We let $\ell_n$ denote the $n$th \emph{Lucas number}, defined recursively by $\ell_{n} = \ell_{n-1} + \ell_{n-2}$ with base cases $\ell_0=2, \ell_1=1$.
\end{defn}

The result of Biggs is the following.

\begin{thm}[Section 9 of \cite{Biggs}]\label{thm:Biggswheel}
For any $n \geq 3$, the critical group of the unsigned wheel is given by:
$$\mathcal{K}(W_n) = \begin{cases}
\mathbb{Z}_{f_n} \oplus \mathbb{Z}_{5f_n}, & \mbox{$n$ is even} \\
\mathbb{Z}_{\ell_n} \oplus \mathbb{Z}_{\ell_n},  & \mbox{$n$ is odd }.\end{cases}$$
\end{thm}

Theorem \ref{thm:wheelresult} is a generalization of Theorem~\ref{thm:Biggswheel} to signed wheel graphs. 
On the way to proving Theorem \ref{thm:wheelresult} we first consider critical groups of negative wheels $-W_n$. Here we get a uniform description that is independent of parity of $n$.

\begin{prop}\label{prop:OddUnbalanced}
For any $n \geq 3$, then we have $|\mathcal{K}(-W_{n})| = 5f_n^2$.

\end{prop}
\begin{proof}
From Theorem \ref{thm:Biggswheel} it follows that the order of the critical group (and hence the number of spanning trees) of the unsigned wheel $W_n$ is $$|\mathcal{K}(W_n)| = \begin{cases}5f_n^2, & \mbox{$n$ is odd} \\ \ell_n^2, & \mbox{n is even.}\end{cases} $$

If $n$ is even, we see that $-W_n$ is switching equivalent to $W_n$ by performing a vertex switch at every other non-sink vertex. From Lemma \ref{lem:SmithBalanced} it follows that $\mathcal{K}(-W_{n})  = \mathcal{K}(W_{n})$  Hence, $|\mathcal{K}(-W_{n})|  = |\mathcal{K}(W_{n})| = 5f_n^2$.

If $n$ is odd, using Theorem \ref{thm:MatrixTreeNice}, we compute $\det(L_{-W_n})$ by counting spanning TU-subgraphs. We see that the set of spanning TU-subgraphs of $-W_n$ consists of the usual spanning trees of $W_n$ and a single subgraph consisting of the exterior cycle of the wheel (along with the isolated vertex $q$). From this count, we obtain that $\det(L_{-W_n}) = \ell_n^2 + 4$. Since $n$ is odd, Identity 53 from \cite{BQ} tells us that $\ell_n^2 + 4 = 5f_{n}^2$.
\end{proof}

Using similar arguments, we can find the order of $\mathcal{K}((W_n)_{\phi})$ where $n \geq 2$ is even and $(W_n)_{\phi}$ is unbalanced.

\begin{prop}\label{prop:EvenUnbalanced}
Suppose $n \geq 2$ is even and $(W_{n})_{\phi}$ is an unbalanced signed wheel. Then we have $|\mathcal{K}((W_{n})_{\phi})| = \ell_n^2$.

\end{prop}
\begin{proof} 
From Lemma \ref{lem: switchingclass}, we conclude that there are two switching equivalence classes of $W_n$. In particular we can assume $(W_n)_\phi$ has a single negative edge incident to two non-sink vertices. 

As in the proof of the previous Proposition, we use Theorem \ref{thm:MatrixTreeNice} and count the spanning TU-subgraphs. As in $-W_n$ for odd $n$ case, the set of spanning TU-subgraphs consists of the usual spanning trees of $W_n$ and a single subgraph consisting of the exterior cycle of the wheel (along with the isolated vertex $q$). From this count, we obtain that $\mbox{det}(L) =5f_n^2 + 4$. Since $n$ is even, Identity 53 from \cite{BQ} tells us that $5f_{n}^2+4 = \ell_n^2$.
\end{proof}

We next establish the structure of the critical group for unbalanced signed graphs $(W_n)_\phi$ for odd $n$. Since by Lemma \ref{lem: switchingclass} there are exactly two switching equivalence classes of $W_n$, it suffices to analyze $\mathcal{K}(-W_{n})$.

\begin{lem}\label{lem:OddCycleUnbalanced}
Let $m \geq 1$. Then we have $$\mathcal{K}(-W_{2m+1}) \cong \mathbb{Z}_{f_{2m+1}} \oplus \mathbb{Z}_{5f_{2m+1}}. $$
\end{lem}
\begin{proof}
For our argument we borrow techniques from Section 3 of \cite{Raza}, where the authors computed the critical group of subdivided wheel graphs.  We first show that the critical group ${\mathcal K}(-W_{2m+1})$ has at most two generators. For this let $L = L_{-W_{2m+1}}$, and suppose $\vec{x} = L\vec{z}$ for some $\vec{z} \in \mathbb{Z}^{2m+1}$. Notice that $x_i + x_{i+2} = -3x_{i+1}$ for all $i$, where the indices are considered modulo $2m+1$. From this, we can express all $x_i$'s as a linear combination of $x_1$ and $x_2$, letting us conclude that ${\mathcal K}(-W_{2m+1})$ has at most 2 generators. Writing $x_j = a_jx_{2} - b_jx_1$ for each $j$, one can check that  $a_1 = 0, b_1 = -1, a_2 = 1, b_2 = 0$, and $a_{j} := (-1)^jf_{2j-2}$ and $b_{j} := (-1)^{j+1}f_{2j-4}$ for $j \geq 3$. Notice we can extend this to $j$ beyond $2m+1$ as well, so we have $x_{1} = x_{2m+2} = a_{2m+2}x_2 - b_{2m+2}x_1$.

Therefore, the Smith normal form of $L$ has $2m-1$ diagonal entries that have value 1. From the previous paragraph, we have $x_{1} = a_{2m+2}x_2 - b_{2m+2}x_1$ and $a_{2m+1}x_2 - b_{2m+1}x_1= x_{2m+1} = -3x_1 - x_2$. This tells us that one only needs to evaluate the Smith normal form of \begin{equation}\label{eq: 2by2}
    A = \begin{pmatrix}
       a_{2m+2} & a_{2m+1}+1 \\ b_{2m+2}+1 & b_{2m+1}-3
    \end{pmatrix} = \begin{pmatrix}
       f_{4m+2}  & -f_{4m}+1 \\ -f_{4m}+1 & f_{4m-2}-3
    \end{pmatrix}.
\end{equation}
Now, to find the Smith normal form of $A$, we find positive integers $d_1, d_2$ such that $d_1^2d_2 = \mbox{det}(A)$ and $d_1 = \mbox{gcd}(f_{4m+2}, f_{4m}-1 , f_{4m-2}-3)$ (c.f. the proof of Theorem 4.1 in \cite{Raza}). The values $d_1$ and $d_1d_2$ are the nontrivial invariant factors of the Smith normal form of $A$ (and hence of $L$). Since all other diagonal entries of the Smith normal form of $L$ is $1$, we also get that $d_1^2d_2 = 5f_{2m+1}^2$ thanks to Proposition \ref{prop:OddUnbalanced}.

To find $d_1$, we utilize Identity 33 from \cite{BQ}, from which we obtain that $f_{4m+2} = f_{2m+1}\ell_{2m+1}$. Furthermore, by consulting Equations 6 and 8 of \cite{BHoggatt1975}, one can also obtain 
\[f_{4m}-1 = f_{4m}-f_2= f_{2m+1}\ell_{2m-1} \text{ and } f_{4m-2}-3 = f_{4m-2}-f_4= f_{2m+1}\ell_{2m-3}.\]
\noindent
From the main theorem of \cite{McDaniel1991DIL}, we have that $\mbox{gcd}(\ell_{2m+1}, \ell_{2m-1}, \ell_{2m-3}) = 1$ for $m \geq 2$, and hence \[\mbox{gcd}(f_{4m+2}, f_{4m}-1 , f_{4m-2}-3) = \mbox{gcd}(f_{2m+1}\ell_{2m+1}, f_{2m+1}\ell_{2m-1}, f_{2m+1}\ell_{2m-3}) = f_{2m+1}.\]

Therefore we get $d_1 = f_{2m+1}$ and $d_2 = 5$. We conclude that $$\mathcal{K}(-W_{2m+1}) = \mathbb{Z}_{f_{2m+1}} \oplus \mathbb{Z}_{5f_{2m+1}}.$$
% Furthermore, using the fact that the product of invariant factors is given by $|\det(L)| = |\mathcal{K}(-W_{2m+1})|$, we have from Proposition \ref{prop:OddUnbalanced} that $d_1^2d_2 = 5f_{2m+1}^2$, which implies $d_2 = 5$. 
\end{proof}

We now study the critical group of the unbalanced signed wheel for even $n = 2m$.

\begin{lem}\label{lem:unbalancedevenwheel}
Suppose $m \geq 2$, and let $(W_{2m})_{\phi}$ be a signed even wheel where all exterior edges except for one are negatively signed. Then, $$\mathcal{K}((W_{2m})_{\phi}) = \mathbb{Z}_{\ell_{2m}} \oplus \mathbb{Z}_{\ell_{2m}}.$$
\end{lem}

\begin{proof}

Without loss of generality, we suppose $(W_{2m})_{\phi}$ is the signed wheel where the only positive edge of the graph is $\{v_2, v_3\}$. The proof will be similar to that of Lemma \ref{lem:OddCycleUnbalanced}.

We first show that the critical group ${\mathcal K}((W_{2m})_{\phi})$ has at most two generators. For this let $L = L_{(W_{2m})_{\phi}}$, and suppose $\vec{x} = L\vec{z}$ for some $\vec{z} \in \mathbb{Z}^{2m}$. As before, notice that $x_i + x_{i+2} = -3x_{i+1}$ for all $i$ except $i=1,2$ (where the indices are considered modulo $2m$). We also have $x_1 - x_3 = -3x_2$ and $-x_2 + x_4 = -3x_3$. From this, we can express all $x_i$'s as a linear combination of $x_1$ and $x_2$, letting us conclude that ${\mathcal K}((W_{2m})_{\phi})$ has at most 2 generators. Writing $x_j = a_jx_{2} - b_jx_1$ for each $j$, one can check that $a_1 = 0, b_1 = -1, a_2 = 1, b_2 = 0$, $ a_3 = 3, b_3 = -1$, and for $j \geq 3$ we have $a_{j} = (-1)^{j+1}f_{2j-2}$ and $b_{j} = (-1)^{j}f_{2j-4}$. As before, we can extend this to $j$ beyond $2m$ as well, so we have $x_{1} = x_{2m+1} = a_{2m+1}x_2 - b_{2m+1}x_1$.

Therefore, the Smith normal form of $L$ has $2m$ diagonal entries that have value 1. From the previous paragraph, we have $-3x_{2m} - x_{2m-1} = a_{2m+1}x_2 - b_{2m+1}x_1$ and $a_{2m}x_2 - b_{2m}x_1= x_{2m} = -3x_1 - x_2$. It then suffices to evaluate the Smith normal form of \begin{equation*}
    A = \begin{pmatrix}
       a_{2m+1} & a_{2m}+1 \\ b_{2m+1}+1 & b_{2m}-3
    \end{pmatrix} = \begin{pmatrix}
       f_{4m}  & -f_{4m-2}+1 \\ -f_{4m-2}+1 & f_{4m-4}-3
    \end{pmatrix}.
\end{equation*}

Now, to find the nontrivial invariant factors of $L$ (and hence the Smith normal form of $A$), as above we need to find positive integers $d_1, d_2$ such that $d_1^2d_2 = |\text{det}(A)|$ and $d_1 = \gcd(f_{4m}, f_{4m-2}-1 , f_{4m-4}-3)$. The $d_1$ and $d_1d_2$ are the nontrivial invariant factors of the Smith normal form of $A$ and of $L$.  Since all other diagonal entries of the Smith normal form of $L$ are $1$, we also get that $d_1^2d_2 = \ell_n^2$ thanks to Proposition~\ref{prop:EvenUnbalanced}.

From Identity 33 from \cite{BQ}, we have $f_{4m} = f_{2m}\ell_{2m}$. Also by Equations 6 and 8 of \cite{BHoggatt1975}, we have $$f_{4m-2} - 1 = f_{4m-2} - (-1)^{2m-2}f_2 = f_{2m-2}\ell_{2m};$$  $$f_{4m-4} - 3 = f_{4m-4} - (-1)^{2m-4}f_4 = f_{2m-4}\ell_{2m}.$$
Hence $d_1 = \gcd(f_{2m}\ell_{2m}, f_{2m-2}\ell_{2m}, f_{2m-4}\ell_{2m})$.
%Hence, $\ell_{2m}$ is a common factor of the entries in $A_2$.
An application of Theorem 6 of \cite{BQ} then gives us
\[\text{gcd}(f_{2m}, f_{2m-2}, f_{2m-4}) = f_{\text{gcd}(2m, 2m-2, 2m-4)} = f_{2} = 1.\]
\noindent
Therefore, $d_1 = \ell_{2m}$ and $d_2 = 1$. We conclude that $$\mathcal{K}((W_{2m})_{\phi}) = \mathbb{Z}_{\ell_{2m}} \oplus \mathbb{Z}_{\ell_{2m}}. $$
\end{proof}

We now combine our previous observations to provide a complete classification of critical groups of signed wheels.

\begin{thm}\label{thm:wheelresult}
.
Then for $n\geq 3$ we have:
$$\mathcal{K}((W_n)_{\phi}) = \begin{cases}
\mathbb{Z}_{f_n} \oplus \mathbb{Z}_{5f_n}, & \mbox{$n$ is odd and $(W_n)_{\phi}$ is unbalanced} \\
\mathbb{Z}_{\ell_n} \oplus \mathbb{Z}_{\ell_n}, & \mbox{$n$ is even and $(W_n)_{\phi}$ is unbalanced}\\\mathbb{Z}_{\ell_n} \oplus \mathbb{Z}_{\ell_n}, & \mbox{$n$ is odd and $(W_n)_{\phi}$ is balanced}\\
\mathbb{Z}_{f_n} \oplus \mathbb{Z}_{5f_n}, & \mbox{$n$ is even and $(W_n)_{\phi}$ is balanced} \end{cases}$$ and this gives all possible critical groups of $(W_n)_{\phi}$.
\end{thm}

\begin{proof} 
Recall that a wheel $W_n$ has two switching equivalence classes. 
 If $(W_n)_\phi$ is balanced, the result follows from Theorem \ref{thm:Biggswheel}, and otherwise we use Lemmas \ref{lem:OddCycleUnbalanced} and \ref{lem:unbalancedevenwheel}.
 \end{proof}

\subsection{Fan Graphs}\label{sec:Fan}
In this section we prove Theorem \ref{thm:FanGroup}, which describes the critical groups of \emph{fan graphs}.  Recall that a wheel graph can be constructed by adding a universal sink to a cycle. If we instead add a universal sink to a path we get the class of fan graphs \cite{brandstadt1999graph}.

\begin{defn}
For an integer $n\geq 3$, the fan graph $F_n$ is defined as the graph obtained by adding a universal sink to a path graph on $n$ vertices.   
\end{defn}

\begin{figure}[h]
 \begin{center}
 
%  \vspace{3.7mm}
    \tikzset{every picture/.style={line width=0.75pt}} %set default line width to 0.75pt        

\begin{tikzpicture}[x=0.75pt,y=0.75pt,yscale=-1,xscale=1]

%Shape: Regular Polygon [id:dp5965108520764912] 
\draw (363,138.5) -- (330.25,195.22);
\draw (330.25,195.22) -- (264.75,195.22);
\draw (264.75,195.22) -- (232,138.5);
\draw (232,138.5) -- (264.75,81.78);
\draw (264.75,81.78)--(330.25,81.78);
\draw (330.25,81.78) -- (363,138.5);
%Straight Lines [id:da6928817439703425] 
\draw    (264.75,195.22) -- (264.75,81.78) ;
%Straight Lines [id:da1883256250171117] 
\draw    (264.75,195.22) -- (330.25,81.78) ;
%Straight Lines [id:da05630781979907873] 
\draw    (264.75,195.22) -- (363,138.5) ;
%Shape: Circle [id:dp867760595013449] 
\draw  [fill={rgb, 255:red, 0; green, 0; blue, 0 }  ,fill opacity=1 ] (228.2,138.5) .. controls (228.2,136.4) and (229.9,134.7) .. (232,134.7) .. controls (234.1,134.7) and (235.8,136.4) .. (235.8,138.5) .. controls (235.8,140.6) and (234.1,142.3) .. (232,142.3) .. controls (229.9,142.3) and (228.2,140.6) .. (228.2,138.5) -- cycle ;
%Shape: Circle [id:dp34800411497245] 
\draw  [fill={rgb, 255:red, 0; green, 0; blue, 0 }  ,fill opacity=1 ] (260.95,195.22) .. controls (260.95,193.12) and (262.65,191.42) .. (264.75,191.42) .. controls (266.85,191.42) and (268.55,193.12) .. (268.55,195.22) .. controls (268.55,197.33) and (266.85,199.03) .. (264.75,199.03) .. controls (262.65,199.03) and (260.95,197.33) .. (260.95,195.22) -- cycle ;
%Shape: Circle [id:dp18329218662298574] 
\draw  [fill={rgb, 255:red, 0; green, 0; blue, 0 }  ,fill opacity=1 ] (326.45,195.22) .. controls (326.45,193.12) and (328.15,191.42) .. (330.25,191.42) .. controls (332.35,191.42) and (334.05,193.12) .. (334.05,195.22) .. controls (334.05,197.33) and (332.35,199.03) .. (330.25,199.03) .. controls (328.15,199.03) and (326.45,197.33) .. (326.45,195.22) -- cycle ;
%Shape: Circle [id:dp19026003444610606] 
\draw  [fill={rgb, 255:red, 0; green, 0; blue, 0 }  ,fill opacity=1 ] (260.95,81.78) .. controls (260.95,79.67) and (262.65,77.97) .. (264.75,77.97) .. controls (266.85,77.97) and (268.55,79.67) .. (268.55,81.78) .. controls (268.55,83.88) and (266.85,85.58) .. (264.75,85.58) .. controls (262.65,85.58) and (260.95,83.88) .. (260.95,81.78) -- cycle ;
%Shape: Circle [id:dp866440108811211] 
\draw  [fill={rgb, 255:red, 0; green, 0; blue, 0 }  ,fill opacity=1 ] (326.45,81.78) .. controls (326.45,79.67) and (328.15,77.97) .. (330.25,77.97) .. controls (332.35,77.97) and (334.05,79.67) .. (334.05,81.78) .. controls (334.05,83.88) and (332.35,85.58) .. (330.25,85.58) .. controls (328.15,85.58) and (326.45,83.88) .. (326.45,81.78) -- cycle ;
%Shape: Circle [id:dp30646924854797275] 
\draw  [fill={rgb, 255:red, 0; green, 0; blue, 0 }  ,fill opacity=1 ] (359.2,138.5) .. controls (359.2,136.4) and (360.9,134.7) .. (363,134.7) .. controls (365.1,134.7) and (366.8,136.4) .. (366.8,138.5) .. controls (366.8,140.6) and (365.1,142.3) .. (363,142.3) .. controls (360.9,142.3) and (359.2,140.6) .. (359.2,138.5) -- cycle ;

% Text Node
\draw (248,195.82) node [anchor=north west][inner sep=0.75pt]    {$q$};
% Text Node
\draw (209,131.4) node [anchor=north west][inner sep=0.75pt]    {$v_{1}$};
% Text Node
\draw (258,62.4) node [anchor=north west][inner sep=0.75pt]    {$v_{2}$};
% Text Node
\draw (323,62.4) node [anchor=north west][inner sep=0.75pt]    {$v_{3}$};
% Text Node
\draw (369,131.4) node [anchor=north west][inner sep=0.75pt]    {$v_{4}$};
% Text Node
\draw (335.25,195.82) node [anchor=north west][inner sep=0.75pt]    {$v_{5}$};

\end{tikzpicture}
    \captionsetup{width=1.0\linewidth}
  \captionof{figure}{The fan Graph $F_5$.}
  \label{fig:fangraph}

 \end{center}
\end{figure}
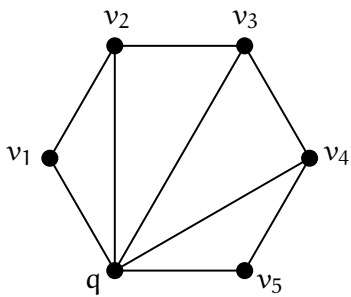

As usual we let $q$ denote the sink, and we use $v_1,\ldots,v_n$ to label the vertices of the underlying path. In particular, $F_n$ is obtained from the wheel graph $W_n$ by removing the edge between $v_1$ and $v_n$. See Figure \ref{fig:fangraph} for an illustration of $F_5$. The reduced Laplacian of $F_n$ is the $n \times n$ matrix given by $$(L_{F_n})_{ij} = \begin{cases} 2, & \mbox{$i=j \in \{1,n\}$} \\ 3, & \mbox{$i = j \notin \{1,n\}$} \\ -1, & \mbox{$|i-j| = 1$} \\ 
%-1, & \mbox{$(i=1, j=2)$ or $(i=n, j = n-1)$} 
0,  & \mbox{otherwise}.  \end{cases}$$

Now suppose $(F_n)_{\phi}$ is any signed fan. 
To prove Theorem \ref{thm:FanGroup} we again first determine the order of $\mathcal{K}((F_n)_{\phi})$.

\begin{lem}\label{lem:OrderFan}
For any signed fan $(F_n)_{\phi}$ we have $|\mathcal{K}((F_n)_{\phi})| = f_{2n}$.

\end{lem}
\begin{proof}
First note that $F_n \backslash q$ is a path, so from Corollary \ref{cor:treegroup} we have $\mathcal{K}((F_n)_{\phi}) \cong \mathcal{K}(|F_n|)$. From the Matrix-Tree Theorem, it then suffices to show that the number of spanning trees of $F_n$ is given by $f_{2n}$. For this we use induction on $n$.

Note that $F_1$ is a single edge and $F_2$ is a triangle, and hence the result holds for $n=1,2$ (recall $f_0 = 0$ and $f_1 = 1$, so that $f_3$ = 3).
%For $n=1,2,3$, a direct calculation of the Smith normal form of the reduced Laplacian yields $\mathcal{K}(|F_1|) = \mathbb{Z}_1 = \mathbb{Z}_{f_{2}}$,  $\mathcal{K}(|F_2|) = \mathbb{Z}_3 = \mathbb{Z}_{f_{4}}$, and $\mathcal{K}(|F_3|) = \mathbb{Z}_{8} = \mathbb{Z}_{f_{6}}.$ Next assume $n \geq 4$ and let $T$ be a spanning tree of $F_n$. 
Now suppose $n \geq 3$ and consider a spanning tree $T$.  We let $e = \{q,v_1\}$ and consider two cases. First suppose $T$ does not contain $e$. Then $T$ consists of the edge $\{v_1, v_2\}$ along with a spanning tree of the graph obtained by removing $v_1$. Note that the subgraph induced by $v_2, ...., v_n, q$ is isomorphic to $F_{n-1}$ and hence by induction there are $f_{2n-2}$ such trees.

Next suppose $T$ does contain $e$. This means that $T$ is the result of adjoining path $q, v_1, ..., v_i$ for some $i$ and a spanning tree of the subgraph induced by $v_{i+1}, ... v_n, q$. Let $i$ be the largest index so that the path $v_1,\ldots,v_i$ is contained in $T$. If $i=n$ then $T$ is itself a path. If $i < n$, when we remove the vertices $v_1,\ldots,v_i$ we obtain a spanning tree of an induced subgraph which is isomorphic to $F_{n-i}$. By induction there are $1+\sum_{k=1}^{n-1}f_{2(n-k)} = f_{2n-1}$ spanning trees of this form.

We conclude that there are $f_{2n-2}+f_{2n-1}= f_{2n}$ spanning trees of $F_n$, which completes the proof.
\end{proof}

We next show that the critical group of a fan graph is cyclic, as usual via a calculation of the Smith normal form of its Laplacian. 

\begin{thm}\label{thm:FanGroup}
Let $(F_n)_{\phi}$ be any signed fan graph with sink given by the vertex of degree $n-1$. Then for $n \geq 1$ we have $\mathcal{K}((F_n)_{\phi}) \cong \mathbb{Z}_{f_{2n}}$.
\end{thm}

\begin{proof} As discussed above, it suffices to determine the critical group $\mathcal{K}(|F_n|)$. From Lemma \ref{lem:OrderFan} it is enough to show that the group is cyclic, i.e. the Smith normal form of $L = L_{|F_{n}|}$ has at most one entry that is larger than 1. For this suppose $\vec{x} = L\vec{z}$ for some $\vec{z} \in \mathbb{Z}^{n}$. From our description of $L_{F_n}$ we have the following relations:
\begin{equation}
\begin{split}
   3x_{i} &= x_{i-1} + x_{i+1}, \mbox{ for all $i \in \{2, 3, ..., n-1\}$}, \\
   2x_1 &= x_2, \\
    2x_{n} &= x_{n-1}.     
    \end{split}
  \end{equation}
    
Thus, for any $i \in \{1, 2, ..., n\}$ there exists an integer $b_i$ such that $x_i = b_ix_1$. Thus the Smith normal form of $L$ has at most one nontrivial diagonal entry, which from Lemma \ref{lem:OrderFan} must be $f_{2n}$. We conclude that $\mathcal{K}(|F_{n}|) \cong \mathbb{Z}_{f_{2n}}$.
\end{proof}

We remark that the computation of ${\mathcal K}(F_n)$ for the unsigned case is an exercise in \cite[Exercise 2.8(e)]{CorryPerk}. Also, while we were preparing this document we were made aware of related results of Selig from \cite{Selig}.  Here the author also considers (unsigned) wheel and fan graphs, and constructs bijections between the set of critical configurations and various combinatorial objects that give an interpretation of the \emph{level} of the configuration. In this context the number of critical configurations of level $k$ can be seen to correspond to a Tutte polynomial evaluation. 
As we have seen, performing a single vertex switching preserves the number of critical configurations and even the isomorphism type of the critical group (see Lemma \ref{lem:SmithBalanced}). However, the level sequence can change and it is an open question to understand the resulting combinatorial structure.

%%%%%%%%%%%%%%%%%%%%%%%%%%%%%

\section{Further Directions}\label{sec:Further}

In this work we have initiated the study of chip-firing on signed graphs, but many natural questions remain.  In this section we discuss some possible further directions of research.

Given a chip-firing pair $(L,M)$, Theorem \ref{thm:KeyIdea} provides a way to understand the $z$-superstable (and critical) configurations of $L$ in terms of the corresponding configurations of $M$. However the precise relationship remains mysterious.

To illustrate this, we consider our running example $G_\phi$ from Section \ref{sec:example}. For the critical configurations $\ccvector{7,6,2}$ and $\ccvector{6,5,2}$, we see that taking the floor of inverse image under $LM^{-1}$ yields the vectors $\ccvector{\frac{2}{3},\frac{11}{6},2}$ and $\ccvector{\frac{2}{3},\frac{4}{3},2}$. Note that taking the floors of each vector yield $\ccvector{0,1,2}$, which (according to Theorem \ref{thm:KeyIdea}) is critical for $G$. On the other hand, one can check that $\ccvector{1,1,2}$ and $\ccvector{1,0,2}$ are critical configurations for $G$ that \emph{cannot} be obtained by floors of inverse images of critical configurations of $G_\phi$.

Similarly, the configurations $\ccvector{0,0,1}$ and $\ccvector{0,1,1}$ and superstable for $G$, but do not arise as floors of inverse images of $z$-superstable configurations of $G_\phi$.  Note that $\ccvector{1,1,2}$ and $\ccvector{0,1,1}$ are in the same equivalence class defined by $M$, and similarly for $\ccvector{1,0,2}$ and $\ccvector{0,0,1}$. Is this always true?

For another, more extreme, example we consider the case of odd negative cycles $-C_{2m+1}$ from Section \ref{sec:cycle}. Here one can see that taking floors of the inverse images of $z$-superstable configurations always results in the zero vector $\vec{0}$. Indeed it is this fact that gives rise to the duality described in Theorem \ref{thm:PetitDuality}. These observations lead to the following question.

\begin{question}
    For a given signed graph $G_{\phi}$, can one describe the set of superstable (resp. critical) configurations of $G$ that one obtains by taking the floor of the inverse image of $z$-superstable (resp. critical) configurations of $G_\phi$? How does this depend on the choice of $\phi$?
\end{question}

For our next question, recall that for unsigned graphs we have a simple duality between critical and superstable configurations (see Section 2.6.5 of \cite{Klivans}).  In particular, for a graph $G$, the configuration $\vec{c}_{\max}$ is a critical configuration with the property that $\vec{d}$ is superstable if and only if $\vec{c}_{\max} - \vec{d}$ is critical.  In Theorem \ref{thm:PetitDuality} we describe a version of duality for a class of signed cycles, but no simple correspondence seems to exist. In general we are looking for a bijection between two equinumerous sets of lattices points in a certain fixed rational cone, which itself seems to be an interesting geometric question.

\begin{question}
Does there exist a duality between the set of $z$-superstable configurations and the set of critical configurations? Can we use Theorem~\ref{thm:KeyIdea} to rephrase the question in terms of duality of the underlying unsigned graph?
\end{question}

Our next question also addresses a desired bijection. We have seen in Lemma \ref{lem:SmithBalanced} that vertex switching preserves the critical group of a signed graph. In particular, if a signed graph is obtained from another via a vertex switching, both signed graphs have the same number of critical and $z$-superstable configurations.  However, it is not clear how the set of configurations themselves change, even in small examples.

To illustrate this, if we again consider our running example $G_\phi$ from Section \ref{sec:example} and switch at the vertex $v_3$, we obtain a new signed graph $H_\psi$ with chip-firing pair $(L,M)$, where 

\[L = \begin{pmatrix} 3 & +1 & +1 \\ +1 & 2 & +1 \\ +1 & +1 & 3 \end{pmatrix}; \quad \quad M = \begin{pmatrix} 3 & -1 & -1 \\ -1 & 2 &-1 \\ -1 & -1 & 3 \end{pmatrix}; \quad \quad LM^{-1} = \begin{pmatrix} \frac{11}{4} & 3 & \frac{9}{4} \\ 2 & 3 & 2 \\ \frac{9}{4} & 3 & \frac{11}{4} \end{pmatrix}. \]

The set of $z$-superstable configurations for $H_\psi$ are given by 
 $$\left\{ \ccvector{5,5,5}, \ccvector{8,6,7}, \ccvector{4,4,4}, \ccvector{0,0,0}, \ccvector{3,3,3}, \ccvector{7,6,8}, \ccvector{9,7,8}, \ccvector{10,8,9}, \ccvector{2,2,2}, \ccvector{1,1,1}, \ccvector{9,8,10}, \ccvector{8,7,9} \right\}.$$

The images under $ML^{-1}$ are

$$\left\{ \ccvector{0,\frac{5}{3},0}, \ccvector{\frac{5}{2},0,\frac{1}{2}}, \ccvector{0,\frac{4}{3},0}, \ccvector{0,0,0}, \ccvector{0,1,0}, \ccvector{\frac{1}{2}, 0, \frac{5}{2}}, \ccvector{\frac{5}{2},\frac{1}{3}, \frac{1}{2}}, \ccvector{\frac{5}{2},\frac{2}{3},\frac{1}{2}}, \ccvector{0,\frac{2}{3},0}, \ccvector{0,\frac{1}{3}, 0}, \ccvector{\frac{1}{2},\frac{2}{3},\frac{5}{2}}, \ccvector{\frac{1}{2},\frac{1}{3},\frac{5}{2}} \right\}.$$ 

If we compare these to the $z$-superstable configurations of the original $G_\phi$, we see that both are sets consisting of $12$ elements, but there does not seem to be a clear relationship. Hence we ask the following.

\begin{question}
    If $G_\psi$ is obtained from $G_\phi$ via a vertex switching, can one describe a bijection between the underlying $z$-superstable (or critical) configurations?
\end{question}

We would also like to understand connections between chip-firing on signed graphs and other combinatorial objects. In the case of chip-firing on a graph $G$,  the superstable configurations coincide with the class of $G$-parking functions (with the case of $G = K_n$ particularly well-studied). By work of Merino \cite{Merino}, the underlying degree sequence (that is, the number of superstable configurations of degree $i$) also relates to Tutte polynomial evaluations, and in particular \emph{activity} of spanning trees. It would be interesting to understand the combinatorial properties of $z$-superstable configurations of signed graphs. 

\begin{question}
For a signed graph $G_\phi$ with chosen sink $q$, does the degree sequence of $z$-superstable configurations (defined in an appropriate way) have any combinatorial meaning?
\end{question}

In a related story, a number of bijections between $G$-parking functions and spanning trees of a graph $G$ have been studied \cite{ChePyl}, generalizing Dhar's algorithm \cite{PhysRevLett.64.1613}. We would like to extend this story to the case of signed graphs, where the set of spanning trees is replaced by the bases of the underlying matroid, and where we ask for a \emph{multijection} in the sense of \cite{McD}.

We next turn to critical groups. We have seen that if $G$ is graph with the property that $G \backslash q$ is a tree, then all choices of signed graphs on $G$ yield the same critical group ${\mathcal K}(G_\phi) \cong {\mathcal K}(G)$. In addition, we have seen that critical groups are invariant under vertex switching. For arbitrary graphs $G$, however, it is not clear how ${\mathcal K}(G_\phi)$ depends on ${\mathcal K}(G)$ or how much structure (if any) is preserved. In our examples, we observe that the number of invariant factors of ${\mathcal K}(G_\phi)$ is equal to that of ${\mathcal K}(G)$.

\begin{question}
For a fixed graph $G$, what are the possible groups ${\mathcal K}(G_\phi)$ that can be obtained as we vary the choice of $\phi$? Does the structure of these groups relate to that of ${\mathcal K}(G)$?
\end{question}

%Is there a parking function ideal for signed graphs? Is there a lattice ideal? If so would have to consider ideals in the semigroup ring determined by $LM^{-1}$.

Finally, we note that many of the constructions and results we discussed in this work can be generalized to the setting of a chip-firing pair $(L,M)$, for an arbitrary $M$-matrix $M$. In the case of signed graphs, we used the reduced Laplacian $M = L_G$ of a graph $G$ and changed the sign of certain off-diagonal entries (corresponding to negative edges) to obtain the invertible matrix $L$.  The same construction can be performed on an arbitrary integral $M$-matrix $M$, and many of our results carry through. For instance, we recover Proposition~\ref{prop:positive} as well as Theorem~\ref{thm:zchigeneral} for the case of $M$-matrices with nonnegative row sums.  

Even more generally, a number of our results follow from Theorem~\ref{thm:KeyIdea}, which holds for chip-firing pairs $(L,M)$, for an \emph{arbitrary} invertible integral matrix $L$ and integral $M$-matrix $M$. This suggests that our results could be generalized to this setting. We see signed graphs as a particularly natural case where one can apply the theory of chip-firing on a pair $(L,M)$, but it would be interesting to apply our results to other settings.

\section*{Acknowledgements}
This research was conducted as part of an REU at Texas State University during Summer 2022, with funding from the NSA and NSF grant \#1757233. We are grateful to these organizations for their financial support and for providing a fruitful work environment. We thank Carly Klivans for sharing her insights on chip-firing, as well as two anonymous referees for helpful comments and corrections on an earlier version of this paper.

\bibliographystyle{plain}
\bibliography{chipfiringsigned_ver2.bib}

\end{document}